\PassOptionsToPackage{hyphens}{url}
\PassOptionsToPackage{hypcap=false}{caption}
\documentclass[a4paper,USenglish,cleveref,autoref,nolineno,thm-restate]{socg-lipics-v2021}

\graphicspath{{./Figures/}}

\usepackage[shortlabels,inline]{enumitem}
\usepackage{float}
\usepackage{amsmath}
\usepackage{tabularx}
\usepackage{overpic}
\usepackage{wrapfig}

\newcommand{\altmanifold}{\mathscr{N}}
\newcommand{\altaltsurface}{\mathcal{R}}
\newcommand{\alttree}{T}

\newcommand{\compbody}{\mathcal{C}}

\newcommand{\decomp}{\mathscr{D}}  

\newcommand{\dual}{\Gamma}

\newcommand{\fork}{F}
\newcommand{\forkcomp}{\mathscr{F}}
\newcommand{\fforkcomp}{\forkcomp_{\circ}}
\newcommand{\nfforkcomp}{\forkcomp_{\color{gray}\bullet}}

\newcommand{\grip}{\gamma}
\newcommand{\hbody}{\mathcal{H}}
\newcommand{\heegaard}{\mathscr{H}}

\newcommand{\jsjdecomp}{\mathscr{D}}

\newcommand{\manifold}{\mathscr{M}}
\newcommand{\maxdeg}{\Delta}

\renewcommand{\root}{\rho}

\newcommand{\surface}{\mathcal{S}}
\newcommand{\altsurface}{\mathbb{S}}

\newcommand{\sweepout}{\Sigma}

\newcommand{\tine}{\tau}
\newcommand{\torus}{\mathbb{T}}

\newcommand{\tri}{\mathscr{T}}

\newcommand{\width}{w}

\newcommand{\heeg}[1]{\mathfrak{g} \left(#1\right)}

\newcommand{\nsphere}[1]{\mathbb{S}^{#1}}

\newcommand{\poly}[1]{\operatorname{poly} (#1)}
\newcommand{\pw}[1]{\operatorname{pw} (#1)}

\newcommand{\tw}[1]{\operatorname{tw} (#1)}

\newcommand\MyFigScale{0.8205128205}
\newcommand\MyFigScaleHuge{0.45128205128}

\newcommand{\myref}[1]{{\color{lipicsGray}\sffamily\bfseries\ref{#1}}}
\definecolor{BrickRed}{rgb}{0.714,0.196,0.11}
\definecolor{Blue}{rgb}{0.176,0.184,0.573}

\newcommand\simtimes{\mathbin{%
    \stackrel{\sim}{\smash{\times}\rule{0pt}{0.9ex}}%
    }}

\theoremstyle{remark}
\newtheorem{question}{Question}

\title{On the Width of Complicated JSJ Decompositions}

\titlerunning{On the Width of Complicated JSJ Decompositions}

\author{Krist\'of Husz\'ar}{Univ Lyon, CNRS, ENS de Lyon, Université Claude Bernard Lyon 1, LIP UMR5668, France \and \url{https://kristofhuszar.github.io/}}{kristof.huszar@ens-lyon.fr}{https://orcid.org/0000-0002-5445-5057}{Supported by the French National Research Agency through the 3IA C\^ote d'Azur (ANR-19-P3IA-0002), AlgoKnot (ANR-20-CE48-0007), GrR (ANR-18-CE40-0032), and TWIN-WIDTH (ANR-21-CE48-0014) projects. A previous version of this manuscript was written during the author's visit at Institut Henri Poincaré (IHP). The author acknowledges the hospitality of IHP (UAR 839 CNRS--Sorbonne Universit\'e) and LabEx CARMIN (ANR-10-LABX-59-01).}

\author{Jonathan Spreer}{School of Mathematics and Statistics F07, The University of Sydney, NSW 2006 Australia \and \url{https://sites.google.com/view/jonathan-spreer/home}}{jonathan.spreer@sydney.edu.au}{https://orcid.org/0000-0001-6865-9483}{Supported by the Australian Research Council's Discovery funding scheme (project no.\ DP220102588).}

\authorrunning{K.\ Husz\'ar and J.\ Spreer}

\Copyright{Krist\'of Husz\'ar and Jonathan Spreer}



\acknowledgements{We are grateful to Arnaud de Mesmay and to Clément Maria for their interest in our work and for inspiring discussions. We thank the anonymous referees for several suggestions to improve this article.}

\nolinenumbers

\hideLIPIcs

\EventEditors{Erin W. Chambers and Joachim Gudmundsson}
\EventNoEds{2}
\EventLongTitle{39th International Symposium on Computational Geometry (SoCG 2023)}
\EventShortTitle{SoCG 2023}
\EventAcronym{SoCG}
\EventYear{2023}
\EventDate{June 12--15, 2023}
\EventLocation{Dallas, Texas, USA}
\EventLogo{socg-logo.pdf}
\SeriesVolume{258}
\ArticleNo{XX}  

\begin{document}

\maketitle

\begin{abstract}
Motivated by the algorithmic study of 3-dimensional manifolds, we explore the structural relationship between the JSJ decomposition of a given 3-manifold and its triangulations. Building on work of Bachman, Derby-Talbot and Sedgwick, we show that a ``sufficiently complicated'' JSJ decomposition of a 3-manifold enforces a ``complicated structure'' for all of its triangulations. More concretely, we show that, under certain conditions, the treewidth (resp.\ pathwidth) of the graph that captures the incidences between the pieces of the JSJ decomposition of an irreducible, closed, orientable 3-manifold $\manifold$ yields a linear lower bound on its treewidth $\tw{\manifold}$  (resp.\ pathwidth $\pw{\manifold}$), defined as the smallest treewidth (resp.\ pathwidth) of the dual graph of any triangulation of $\manifold$.

We present several applications of this result. We give the first example of an infinite family of bounded-treewidth 3-mani\-folds with unbounded pathwidth. We construct Haken 3-mani\-folds with arbitrarily large treewidth---previously the existence of such 3-manifolds was only known in the non-Haken case. We also show that the problem of providing a constant-factor approximation for the treewidth (resp.\ pathwidth) of bounded-degree graphs efficiently reduces to computing a constant-factor approximation for the treewidth (resp.\ pathwidth) of 3-manifolds. 
\end{abstract}

\section{Introduction}
\label{sec:intro}

Manifolds in geometric topology are often studied through the following two-step process. Given a piecewise linear $d$-dimensional manifold $\manifold$, first find a ``suitable'' triangulation of it, i.e., a decomposition of $\manifold$ into $d$-simplices with ``good'' combinatorial properties. Then apply algorithms on this triangulation to reveal topological information about $\manifold$.

The work presented in this article is motivated by this process in dimension $d=3$. Here every manifold can be triangulated \cite{moise1952affine} and questions about them typically admit algorithmic solutions \cite{kuperberg2019algorithmic, lackenby2020algorithms, scott2014homeomorphism}.\footnote{In higher dimensions none of these statements is true in general. See, e.g., \cite{manolescu2016lectures}, \cite{markov1958homeo} or \cite[Section 7]{poonen2014undecidable}.} At the same time, the feasibility of a particular computation can greatly depend on structural properties of the triangulation in use.
Over the past decade, this phenomenon was recognized and exploited in various settings, leading to \emph{fixed-parameter tractable} (FPT) algorithms for several problems in low-dimensional topology, some of which are even known to be \textbf{NP}-hard \cite{burton2017courcelle, burton2016parameterized, burton2018algorithms, pettersson2014fixed, burton2013complexity}.\footnote{See \cite{bagchi2016efficient} for an FPT algorithm checking tightness of (weak) pseudomanifolds in arbitrary dimensions.} Although these algorithms have exponential running time in the worst case, for input triangulations with \emph{dual graph} of bounded \emph{treewidth} they always terminate in polynomial (in most cases, linear) time.\footnote{The running times are given in terms of the \emph{size} of the input triangulation, i.e., its number of tetrahedra.} Moreover, some of them have implementations that are highly effective in practice, providing useful tools for researchers in low-dimensional topology \cite{burton2013regina, Regina}.

The theoretical efficiency of the aforementioned FPT algorithms crucially depends on the assumption that the dual graph of the input triangulation has small treewidth. To understand their scope, it is thus instructive to consider the \emph{treewidth $\tw{\manifold}$ of a compact $3$-manifold $\manifold$}, defined as the smallest treewidth of the dual graph of any triangulation of $\manifold$. Indeed, the relation between the treewidth and other quantities associated with 3-manifolds has recently been investigated in various contexts \cite{huszar2020combinatorial,huszar2022pathwidth,huszar2019manifold,huszar2019treewidth,maria2019treewidth}. For instance, in \cite{huszar2019treewidth} together with Wagner we have shown that the treewidth of a \emph{non-Haken} 3-manifold is always bounded below in terms of its \emph{Heegaard genus}. Combined with earlier work of Agol \cite{agol2003small}---who constructed non-Haken 3-manifolds with arbitrary large Heegaard genus---this implies the existence of 3-manifolds with arbitrary large treewidth. Despite the fact that, asymptotically, most triangulations of most 3-manifolds must have dual graph of large treewidth \cite[Appendix A]{huszar2019treewidth}, this collection described by Agol has remained, to this date, the only known family of 3-manifolds with arbitrary large treewidth.

\subparagraph*{The main result} In this work we unravel new structural links between the triangulations of a given 3-manifold and its \emph{JSJ decomposition} \cite{jaco1978decomposition,jaco1979seifert,johannson1979homotopy}. Employing the machinery of \emph{generalized Heegaard splittings} \cite{scharlemann2016lecture}, the results developed in \cite{huszar2019treewidth}, and building on the work of Bachman, Derby-Talbot and Sedgwick \cite{bachman2016heegaard,bachman2017computing}, we show that, under suitable conditions, the dual graph of any triangulation of a given 3-manifold $\manifold$ inherits structural properties from the \emph{decomposition graph} that encodes the incidences between the pieces of the JSJ decomposition of $\manifold$. More precisely, in \Cref{sec:main} we prove the following theorem.

\begin{theorem}[Width inheritance]\label{thm:jsj-width}
For any closed, orientable and irreducible $3$-manifold $\manifold$ with \emph{sufficiently complicated}\footnote{The notion of ``sufficiently complicated'' under which we establish Theorem \ref{thm:jsj-width} is discussed in \Cref{sec:main}.} torus gluings in its JSJ decomposition $\decomp$, the treewidth and pathwidth of $\manifold$ and that of the decomposition graph $\dual(\decomp)$ of $\decomp$ satisfy
\vspace{-.5\topsep}

\noindent
\begin{minipage}[b]{.425\textwidth}
\begin{equation}\label{eq:jsj-tw}
	\tw{\dual(\decomp)} \leq 18\cdot(\tw{\manifold}+1)
\end{equation}
\end{minipage}\hfill
\begin{minipage}[b]{.1\textwidth}
\centering
\quad and
\end{minipage}\hfill
\begin{minipage}[b]{.425\textwidth}
\begin{equation}\label{eq:jsj-pw}
	 \pw{\dual(\decomp)} \leq 4\cdot(3\pw{\manifold}+1).
\end{equation}
\end{minipage}
\end{theorem}

\subparagraph*{An algorithmic construction} Much work in 3-dimensional topology has been devoted to the study of 3-manifolds constructed by pasting together simpler pieces along their boundary surfaces via ``sufficiently complicated'' gluing maps, and to understand how different decompositions of the same 3-manifold interact under various conditions, see, e.g., \cite{bachman2013stabilizing, bachman2016heegaard, bachman2017computing, bachman2006sweepouts, lackenby2004heegaard, li2010heegaard, scharlemann2001comparing, schultens2007destabilizing}.
\Cref{thm:jsj-width} allows us to leverage these results to construct 3-manifolds, where we have tight control over the treewidth and pathwidth of their triangulations \cite{huszar2020combinatorial}.

By combining \Cref{thm:jsj-width} and \cite[Theorem 5.4]{bachman2013stabilizing} (cf.\ \cite[Appendix]{bachman2017computing}), in \Cref{sec:const} we prove the following result.

\begingroup

\begin{theorem}\label{thm:construction}
There is a polynomial-time algorithm that, given an $n$-node graph $G$ with maximum node-degree $\maxdeg$, produces a triangulation $\tri_G$ of a closed $3$-manifold $\manifold_G$, such that
\begin{enumerate}
	\item the triangulation $\tri_G$ contains $O_\maxdeg(\pw{G}\cdot n)$ tetrahedra,\footnote{Similar to the standard \emph{big-O notation}, $O_\maxdeg(x)$ means ``a quantity bounded above by $x$ times a constant depending on $\Delta$.'' To ensure that \myref{eq:tw-inherit} is satisfied, but not necessarily \myref{eq:pw-inherit}, $O_\maxdeg(\tw{G}\cdot n)$ tetrahedra suffice.} \label{thm:construction-size}
	\item the JSJ decomposition $\mathscr{D}$ of $\manifold_G$ satisfies $\dual(\mathscr{D}) = G$, and  \label{thm:construction-jsj}
	\item there exist universal constants $c,c'>0$, such that
	\begin{enumerate}
		\item $(c / \maxdeg) \tw{\manifold_G} \leq \tw{G} \leq 18\cdot(\tw{\manifold_G}+1)$, and \label{eq:tw-inherit}
		\item $(c' / \maxdeg) \pw{\manifold_G} \leq \pw{G} \leq 4\cdot(3\pw{\manifold_G}+1)$. \label{eq:pw-inherit}
	\end{enumerate}
\end{enumerate}
\end{theorem}

\endgroup

\subparagraph*{Applications} In \Cref{sec:appl} we present several applications of \Cref{thm:construction}. First, we construct a family of bounded-treewidth 3-manifolds with unbounded pathwidth (\Cref{cor:unbounded-pathwidth}). Second, we exhibit Haken 3-manifolds with arbitrary large treewidth (\Cref{cor:unbounded-treewidth}). To our knowledge, no such families of 3-manifolds had been known before. Third, we show that the problem of providing a constant-factor approximation for the treewidth (resp.\ pathwidth) of bounded-degree graphs reduces in polynomial time to computing a constant-factor approximation for the treewidth (resp.\ pathwidth) of 3-manifolds (\Cref{cor:hardness-treewidth}). This reduction, together with previous results \cite{raghavendra2010graph,wu2014inapproximability, yamazaki2018inapproximability}, suggests that this problem may be computationally hard.

\begin{proof}[Outline of the proof of Theorem \ref{thm:jsj-width}]
We now give a preview of the proof of our main result. As the arguments for showing \eqref{eq:jsj-tw} and \eqref{eq:jsj-pw} are analogous, we only sketch the proof of \eqref{eq:jsj-tw}. To show that $\tw{\dual(\jsjdecomp)} \leq 18(\tw{\manifold}+1)$, we start with a triangulation $\tri$ of $\manifold$ whose dual graph has minimal treewidth, i.e., $\tw{\dual(\tri)} = \tw{\manifold}$. Following \cite[Section~6]{huszar2019treewidth}, we construct from $\tri$ a generalized Heegaard splitting $\heegaard$ of $\manifold$, together with a \emph{sweep-out} $\sweepout=\{\sweepout_x : x \in H\}$ along a tree $H$, such that the genus of any level surface $\sweepout_x$ is at most $18\cdot(\tw{\dual(\tri)}+1)$. If $\heegaard$ is not already \emph{strongly irreducible}, we repeatedly perform \emph{weak reductions} until we get a strongly irreducible generalized Heegaard splitting $\heegaard'$ with associated \emph{sweep-out} $\sweepout'=\{\sweepout'_x : x \in H\}$ along the same tree $H$. Crucially, weak reductions do not increase the genera of level surfaces \cite[Section 5.2]{scharlemann2016lecture}, thus $18\cdot(\tw{\dual(\tri)}+1)$ is still an upper bound on those in $\sweepout'$. Now, by the assumption of the JSJ decomposition of $\manifold$ being ``sufficiently complicated,'' each \emph{JSJ torus} can be isotoped in $\manifold$ to coincide with a connected component of some \emph{thin level} of $\heegaard'$. This implies that, after isotopy, each level set $\sweepout'_x$ is incident to at most $18\cdot(\tw{\dual(\tri)}+1)+1$ JSJ pieces of $\manifold$. Sweeping along $H$, we can construct a \emph{tree decomposition} of $\dual(\jsjdecomp)$ where each \emph{bag} contains at most $18\cdot(\tw{\dual(\tri)}+1)+1$ nodes of $\dual(\jsjdecomp)$. Hence $\tw{\dual(\jsjdecomp)} \leq 18\cdot(\tw{\dual(\tri)}+1) = 18\cdot(\tw{\manifold}+1)$.
\end{proof}

\subparagraph*{Organization of the paper} In \Cref{sec:prelims} we review the necessary background on graphs and 3-manifolds. \Cref{sec:gen} contains a primer on generalized Heegaard splittings, which provides us with the indispensable machinery for proving our main result (\Cref{thm:jsj-width}) in \Cref{sec:main}. In \Cref{sec:const} we describe the algorithmic construction of 3-manifolds that ``inherit'' their combinatorial width from that of their JSJ decomposition graph (\Cref{thm:construction}). Then, in \Cref{sec:appl} we present the aforementioned applications of this construction (\Cref{cor:unbounded-pathwidth,cor:unbounded-treewidth,cor:hardness-treewidth}). The paper is concluded with a discussion and some open questions in \Cref{sec:discussion}. Selected results from 3-manifold topology we rely on are collected in the \hyperref[app]{Appendix}.

\section{Preliminaries}
\label{sec:prelims}

\subsection{Graphs}
\label{ssec:graphs} 

A {\em (multi)graph} $G=(V,E)$ is a finite set $V=V(G)$ of {\em nodes}\footnote{Throughout this paper we use the terms {\em edge} and {\em vertex} to refer to an edge or vertex in a $3$-manifold triangulation, whereas the terms {\em arc} and {\em node} denote an edge or vertex in a graph, respectively.} together with a multiset $E=E(G)$ of unordered pairs of not necessarily distinct nodes, called {\em arcs}. The {\em degree} $d_v$ of a node $v \in V$ equals the number of arcs containing it, where loop arcs are counted twice. $G$ is {\em $k$-regular}, if $d_v = k$ for all $v \in V$. A \emph{tree} is a connected graph with $n$ nodes and $n-1$ arcs. 

A {\em tree decomposition} of a graph $G=(V,E)$ is a pair $(\mathscr{X}=\{B_i:i \in I\},\alttree=(I,F))$ with {\em bag}s $B_i \subseteq V$ and a tree $\alttree=(I,F)$, such that
 \begin{enumerate*}
	\item $\bigcup_{i \in I} B_i = V$ (node coverage),
	\item for all arcs $\{u,v\} \in E$, there exists $i \in I$ such that $\{u,v\} \subseteq B_i$ (arc coverage), and
	\item for all $v \in V$, $T_v = \{i \in I:v \in B_i\}$ spans a connected sub-tree of $\alttree$ (sub-tree property).
\end{enumerate*}
The \textit{width} of a tree decomposition equals $\max_{i \in I}|B_i|-1$, and the \emph{treewidth} $\tw{G}$ is the smallest width of any tree decomposition of $G$. Replacing all occurrences of ``tree'' with ``path'' in the definition of treewidth yields the notion of {\em pathwidth} $\pw{G}$. We have $\tw{G} \leq \pw{G}$.

\subsection{Manifolds}
\label{ssec:manifolds}

A {\em $d$-dimensional manifold} is a topological space $\manifold$, where each point $x \in \manifold$ has a neighborhood homeomorphic to $\mathbb{R}^d$ or to the closed upper half-space $\{(x_1,\ldots,x_d) \in \mathbb{R}^d : x_d\geq 0\}$. The latter type of points of $\manifold$ constitute the {\em boundary $\partial \manifold$} of $\manifold$. A compact manifold is said to be {\em closed} if it has an empty boundary. We consider manifolds up to homeomorphism (``continuous deformations'') and write $\manifold_1 \cong \manifold_2$ for homeomorphic manifolds $\manifold_1$ and $\manifold_2$.

\paragraph*{3-Manifolds and surfaces}
\label{ssec:3manifolds}

The main objects of study in this paper are 3-dimensional manifolds, or \emph{$3$-manifolds} for short. Here we give a brief introduction to $3$-manifolds tailored to our purposes. We refer the reader to \cite{schultens2014introduction} (and the references therein) for more details. All 3-manifolds and surfaces encountered in this article are compact and orientable.
We let $\altsurface_g$ denote the closed, connected, orientable surface of genus $g$. We also refer to the $d$-dimensional torus and sphere as $\torus^d$ and $\nsphere{d}$, respectively (hence $\altsurface_0 = \nsphere{2}$ and $\altsurface_1 = \torus^2$). The genus $g(\surface)$ of a (not necessarily connected) surface $\surface$ is defined to be the sum of the genera of its connected components. 

\subparagraph*{Triangulations and the treewidth of 3-manifolds} A {\em triangulation} $\tri$ of a given $3$-manifold $\manifold$ is a finite collection of abstract tetrahedra glued together along pairs of their triangular faces, such that the resulting space is homeomorphic to $\manifold$. Unpaired triangles comprise a triangulation of the boundary of $\manifold$. Note that the face gluings may also identify several tetrahedral edges (or vertices) in a single {\em edge} (or {\em vertex}) of $\tri$. Every compact 3-manifold admits a triangulation \cite{moise1952affine} (cf.\ \cite{bing1959alternative}). Given a triangulation $\tri$, its {\em dual graph} $\dual (\tri)$ is the multigraph whose nodes correspond to the tetrahedra in $\tri$, and arcs to face gluings (\Cref{fig:tetrahedra}).

\begin{figure}[ht]
	\centerline{\includegraphics[scale=\MyFigScale]{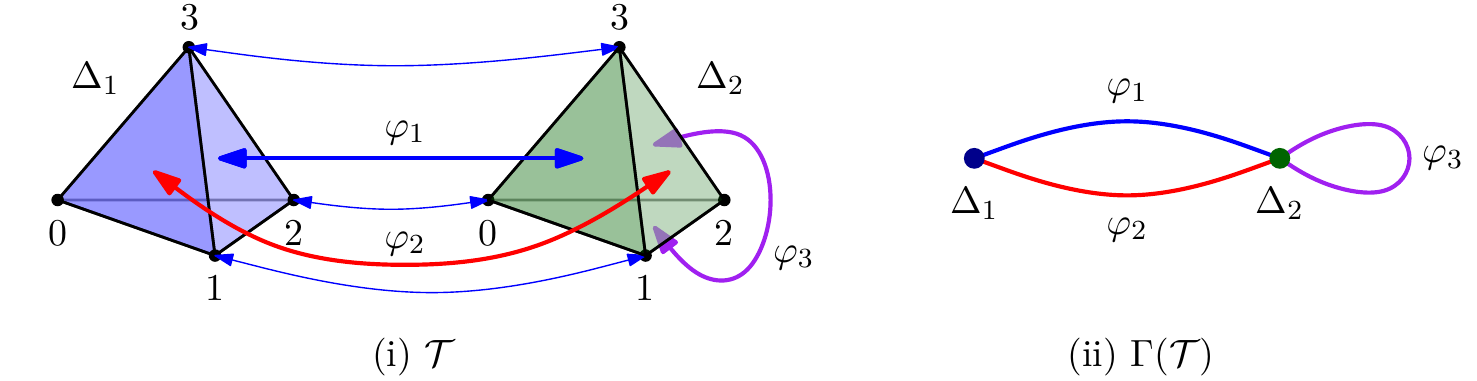}}
	\caption{(i) Example of a triangulation $\tri$ with two tetrahedra $\Delta_1$ and $\Delta_2$, and three face gluing maps $\varphi_1$, $\varphi_2$ and $\varphi_3$. The map $\varphi_1$ is specified to be $\Delta_1 (123) \longleftrightarrow\Delta_2 (103)$. (ii) The dual graph $\Gamma (\tri)$ of the triangulation $\tri$. Reproduced from \cite[Figure 1]{huszar2022pathwidth}.}
	\label{fig:tetrahedra}
\end{figure}

For a compact 3-manifold $\manifold$, its \emph{treewidth} $\tw{\manifold}$ (resp.\ \emph{pathwidth} $\pw{\manifold}$) is defined as the smallest treewidth (resp.\ pathwidth) of the dual graph of any triangulation of $\manifold$.

\subparagraph*{Incompressible surfaces and essential disks} Given a $3$-manifold $\manifold$, a surface $\surface \subset \manifold$ is said to be \emph{properly embedded} in $\manifold$ if $\partial\surface \subset \partial\manifold$ and $(\surface \setminus \partial\surface) \subset (\manifold \setminus \partial\manifold)$. Given a properly embedded surface $\surface \subset \manifold$, an embedded disk $D\subset \manifold$ with  $\operatorname{int}(D) \cap \surface = \emptyset$ and $\partial D \subset \surface$ a curve not bounding a disk on $\surface$ is called a {\em compressing disk}. If such a disk exists, $\surface$ is said to be {\em compressible} in $\manifold$, otherwise -- and if $\surface$ is not a $2$-sphere -- it is called {\em incompressible}. A $3$-manifold $\manifold$ is said to be {\em irreducible}, if every embedded $2$-sphere bounds a $3$-ball in $\manifold$. A disk $D \subset \manifold$ properly embedded in a 3-manifold $\manifold$ is called \emph{inessential} if it cuts off a 3-ball from $\manifold$, otherwise $D$ is called \emph{essential}. A compact, orientable, irreducible $3$-manifold is called {\em Haken} if it contains an orientable, properly embedded, incompressible surface, and otherwise is referred to as {\em non-Haken}.

\subparagraph*{Heegaard splittings of closed 3-manifolds} A {\em handlebody} $\hbody$ is a connected $3$-manifold homeomorphic to a thickened graph. The {\em genus $g(\hbody)$ of $\hbody$} is defined as the genus of its boundary surface $\partial \hbody$.  A {\em Heegaard splitting} of a closed, orientable 3-manifold $\manifold$ is a decomposition $\manifold = \hbody \cup_\surface \hbody'$ where $\hbody$ and $\hbody'$ are homeomorphic handlebodies with $\hbody \cup \hbody' = \manifold$ and $\hbody \cap \hbody' = \partial\hbody = \partial\hbody' = \surface$ called the {\em splitting surface}. Introduced in  \cite{heegaard1916analysis}, the {\em Heegaard genus} $\heeg{\manifold}$ of $\manifold$ is the smallest genus $g(\surface)$ over all Heegaard splittings of $\manifold.$

\subparagraph*{The JSJ decomposition} A central result by Jaco--Shalen \cite{jaco1978decomposition,jaco1979seifert} and Johannson \cite{johannson1979homotopy} asserts that every closed, irreducible and orientable 3-manifold $\manifold$ admits a collection $\mathbf{T}$ of pairwise disjoint embedded, incompressible tori, where each piece of the complement $\manifold \setminus \mathbf{T}$ is either Seifert fibered\footnote{See \cite[Section 3.7]{schultens2014introduction} or \cite[p.\ 18]{hatcher2007notes} for an introduction to Seifert fibered spaces (cf.\ \Cref{app:jsj}).} or atoroidal\footnote{An irreducible 3-manifold $\manifold$ is called {\em atoroidal} if every incompressible torus in $\manifold$ is boundary-parallel.}. A minimal such collection of  tori is unique up to isotopy and gives rise to the so-called {\em JSJ decomposition} (or {\em torus decomposition}) of $\manifold$ \cite[Theorem 1.9]{hatcher2007notes}. We refer to this collection of incompressible tori as the {\em JSJ tori} of $\manifold$. The graph with nodes the JSJ pieces, and an arc for each JSJ torus (with endpoints the two nodes corresponding to its two adjacent pieces) is called the {\em dual graph $\dual(\jsjdecomp)$ of the JSJ decomposition $\jsjdecomp$ of $\manifold$}.

\section{Generalized Heegaard Splittings}
\label{sec:gen}

A Heegaard splitting of a closed 3-manifold is a decomposition  into two identical handlebodies along an embedded surface. Introduced by Scharlemann and Thompson \cite{scharlemann1992thin}, a generalized Heegaard splitting of a compact 3-manifold $\manifold$ (possibly with boundary) is a decomposition of $\manifold$ into several pairs of compression bodies along a family of embedded surfaces, subject to certain rules. Following \cite[Chapters 2 and 5]{scharlemann2016lecture},\footnote{For an open-access version, see \cite[Sections 3.1 and 4]{scharlemann2005lecture}.} here we give an overview of this framework.

\subsection{Compression bodies and forks}

A \emph{compression body} $\compbody$ is a 3-manifold with boundary obtained by the following procedure. \begin{enumerate}
	\item Consider the thickening $\surface \times [0,1]$ of  a closed, orientable, possibly disconnected surface $\surface$,\label{compbody:1}
	\item optionally attach some \emph{$1$-handles}, each being of the form $D \times [0,1]$ (thickened edge, where the disk $D$ is the cross-section), to $\surface \times \{1\}$ along $D\times \{0\} \cup D\times \{1\}$,\label{compbody:2} and
	\item optionally fill in some collection of $2$-sphere components of $\surface \times \{0\}$ with $3$-balls.\label{compbody:3}
\end{enumerate}
We call $\partial_{-}\compbody = \surface \times \{0\} \setminus \{\text{filled-in 2-sphere components}\}$ the {\em lower boundary} of $\compbody$ and $\partial_{+}\compbody = \partial\,\compbody \setminus \partial_-\compbody$ its {\em upper boundary}. By construction, $g(\partial_{+}\compbody) \geq g(\partial_{-}\compbody)$. Note that, if $\partial_{-}\compbody$ is empty, then $\compbody$ is a handlebody. We allow compression bodies to be disconnected.

A \emph{fork}, more precisely an \emph{$n$-fork}, $\fork$ is a tree with $n+2$ nodes $V(\fork)=\{\root,\grip,\tine_1,\ldots,\tine_{n}\}$, where $\root$, called the \emph{root}, is of degree $n+1$, and the other nodes are leaves. One of them, denoted $\grip$, is called the \emph{grip}, and the remaining leaves $\tine_1,\ldots,\tine_n$ are called \emph{tines}. A fork can be regarded as an abstraction of a connected compression body $\compbody$, where the grip corresponds to $\partial_+\compbody$ and each tine corresponds to a connected component of $\partial_-\compbody$, see, e.g.,\ \Cref{fig:compbody}.

\begin{figure}[ht]
\centering
     \begin{subfigure}[b]{0.375\textwidth}
         \centering
         \begin{overpic}[scale=.9]{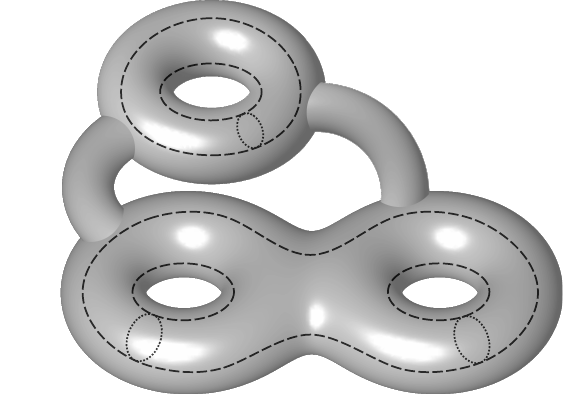}
		\put (3.5,40) {\small{$h_1$}}
		\put (75,47) {\small{$h_2$}}
		\put (56,64) {\small{$\torus^2 \times [0,1]$}}
		\put (41,-6) {\small{$\altsurface_2 \times [0,1]$}}
         \end{overpic}
         \vspace{12pt}
         \caption{A compression body $\compbody$}
         \label{fig:compbody-construction}
     \end{subfigure}
     \hfill
     \begin{subfigure}[b]{0.3\textwidth}
         \centering
         \includegraphics[scale=.9]{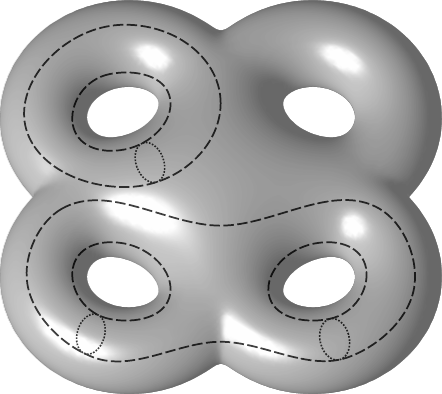}
         \vspace{12pt}
         \caption{$\compbody$ after some isotopy}
         \label{fig:compbody-isotopy}
     \end{subfigure}
     \hfill
     \begin{subfigure}[b]{0.26\textwidth}
         \centering
         \begin{overpic}[scale=.9]{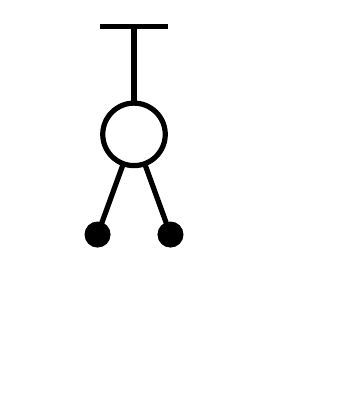}
         		\put (5,91) {\small{$\partial_+\compbody$}}
		\put (46.5,91) {\small{$\leftarrow\textit{grip}$}}
		\put (46.5,64) {\small{$\leftarrow\textit{root}$}}
		\put (48,48.5) {\small{\raisebox{-6.5pt}{\rotatebox{40}{$\leftarrow$}}\,\textit{tine}}}
         		\put (2,25) {\small{$\partial^{(1)}_-\compbody$}}
         		\put (46,25) {\small{$\partial^{(2)}_-\compbody$}}
		\put (55,13) {\small{\rotatebox{90}{$=$}}}
		\put (11,13) {\small{\rotatebox{90}{$=$}}}
		\put (1,3) {\small{$\torus^2 \times\{0\}$}}
		\put (41.25,3) {\small{$\altsurface_2\times\{0\}$}}
         \end{overpic}
         \vspace{12pt}
         \caption{A $2$-fork representing $\compbody$}
         \label{fig:compbody-fork}
     \end{subfigure}
        \caption{The compression body $\compbody$ is obtained by first thickening the disconnected surface $\torus^2 \cup \altsurface_2$ to $(\torus^2 \cup \altsurface_2)\times[0,1]$, then attaching two $1$-handles ($h_1$ and $h_2$) between $\torus^2 \times \{1\}$ and $\altsurface_2\times \{1\}$. For the lower boundary of $\compbody$ we have $\partial_-\compbody=(\torus^2 \cup \altsurface_2)\times\{0\}$, and for its upper boundary $\partial_+\compbody=\partial\,\compbody\setminus\partial_-\compbody \cong \altsurface_4.$}
        \label{fig:compbody}
\end{figure}

\subparagraph*{Non-faithful forks} In certain situations (notably, in the proof of \Cref{thm:jsj-width}, cf.\ \Cref{sec:main}) it is useful to also take a simplified view on a generalized Heegaard splitting. To that end, one may represent several compression bodies by a single \emph{non-faithful} fork, where the grip and the tines may correspond to collections of boundary components. To distinguish faithful forks from non-faithful ones, we color the roots of the latter with gray, see \Cref{fig:faithful-forks}.

\bigskip

\begin{nolinenumbers}
\noindent
\begin{minipage}{0.55\textwidth}
	\centering
	\begin{overpic}[scale=.9]{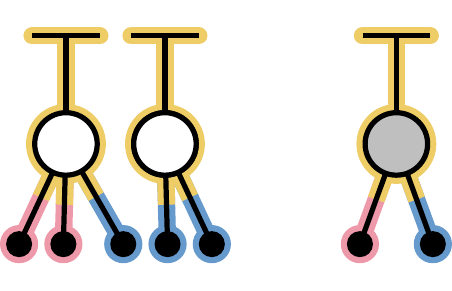}
		\put (57.5,29.85) {\Large$\leadsto$}
	\end{overpic}
	\captionof{figure}{Two faithful forks bundled into a non-faithful fork. The colors show the grouping of the tines.}
	\label{fig:faithful-forks}
\end{minipage}\hfill%
\begin{minipage}{0.36\textwidth}
	\centering
	\begin{overpic}[scale=.9]{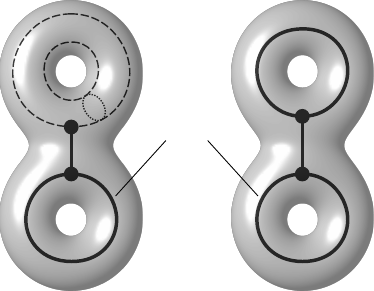}
		\put (47.25,40.25) {$\Gamma$}
		\put (-10,1) {$\compbody$}
		\put (102,1) {$\hbody$}
	\end{overpic}
	\captionof{figure}{\begin{nolinenumbers}Spines of a compression body $\compbody$ and of a handlebody $\hbody$.\end{nolinenumbers}}
	\label{fig:compbody-spine}
\end{minipage}
\end{nolinenumbers}

\subparagraph*{Spines and sweep-outs of compression bodies} A graph $\Gamma$ embedded in a compression body $\compbody$ is called a \emph{spine} of $\compbody$, if every node of $\Gamma$ that is incident to $\partial_-\compbody$ is of degree one, and $\compbody \setminus (\Gamma \cup\partial_-\compbody) \cong \partial_+\compbody \times (0,1]$, see \Cref{fig:compbody-spine}. Assume first that $\partial_-\compbody\neq\emptyset$. A \emph{sweep-out of $\compbody$ along an interval}, say $[-1,1]$, is a continuous map $f\colon\compbody\rightarrow[-1,1]$, such that $f^{-1}(\pm1) = \partial_{\pm}\compbody$,
\begin{align}
	f^{-1}(t) \cong\begin{cases} \partial_{+}\compbody, &\text{if} \phantom{-}t \in (0,1),\\  \Gamma \cup\partial_-\compbody, &\text{if} \phantom{-}t=0,~\text{and} \\ \partial_{-}\compbody, &\text{if} \phantom{-}t \in (-1,0).  \end{cases}
	\label{eq:sweep-out}
\end{align}
Writing $\Sigma_t = f^{-1}(t)$, we get a 1-parameter family of surfaces (except for $\Sigma_0$) ``sweeping through'' $\compbody$. For handlebodies, the definition of a sweep-out is similar, but it ``ends at $0$'' with the spine: $f^{-1}(1) = \partial\hbody$, $f^{-1}(t) \cong \partial\hbody$ if $t \in (0,1)$, and $f^{-1}(0) = \Gamma$, see \Cref{fig:sweep-out}.

\begin{figure}[ht]
     \begin{subfigure}[b]{0.48\textwidth}
         \centering
         \begin{overpic}[scale=.9]{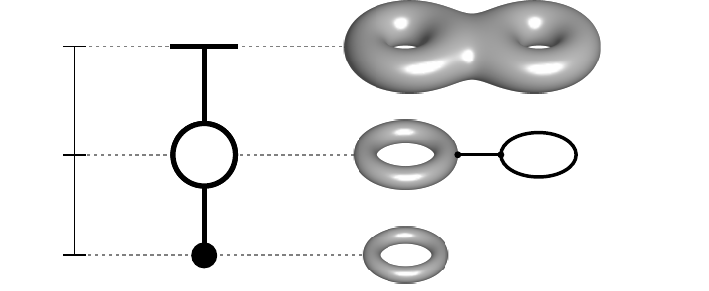}
		\put (0,32.75) {\small{$\phantom{-}1\phantom{-}$}}
		\put (0,17.5) {\small{$\phantom{-}0\phantom{-}$}}
		\put (0,3.25) {\small{$-1\phantom{-}$}}
		\put (88,32.75) {\small{$\Sigma_{1}$}}
		\put (88,17.5) {\small{$\Sigma_{0}$}}
		\put (88,3.25) {\small{$\Sigma_{-1}$}}
         \end{overpic}
     \end{subfigure}
     \begin{subfigure}[b]{0.48\textwidth}
         \centering
         \begin{overpic}[scale=.9]{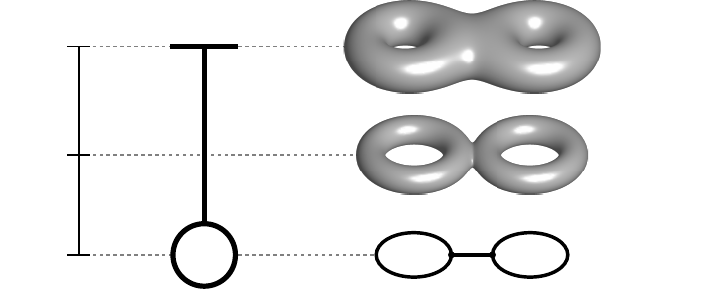}
		\put (0,32.75) {\small{$\phantom{1/}1$}}
		\put (0,17.5) {\small{$1/2$}}
		\put (0,3.25) {\small{$\phantom{1/}0$}}
		\put (88,32.75) {\small{$\Sigma_{1}$}}
		\put (88,17.5) {\small{$\Sigma_{1/2}$}}
		\put (88,3.25) {\small{$\Sigma_{0}$}}
         \end{overpic}
     \end{subfigure}
     \caption{Sweep-outs of the compression body $\compbody$ and of the handlebody $\hbody$ shown in \Cref{fig:compbody-spine}.}
     \label{fig:sweep-out}
\end{figure}

\begin{remark}[Sweep-out along a fork]
\label{rem:sweep-out-fork}
It is straightforward to adapt the definition of sweep-out in a way that, instead of an interval, we sweep a compression body $\compbody$ along a (faithful or non-faithful) fork $\fork$. Such a sweep-out is a continuous map $f\colon\compbody\rightarrow\|\fork\|$ (where $\|\fork\|$ denotes a geometric realization of $\fork$) that satisfies very similar requirements to those in \eqref{eq:sweep-out}, however, the components of the lower boundary $\partial_-\compbody$ may appear in different level sets, see \Cref{fig:sweep-out-fork}.
\end{remark}

\begin{figure}[ht]
     \begin{subfigure}[b]{0.48\textwidth}
         \centering
         \begin{overpic}[scale=.9]{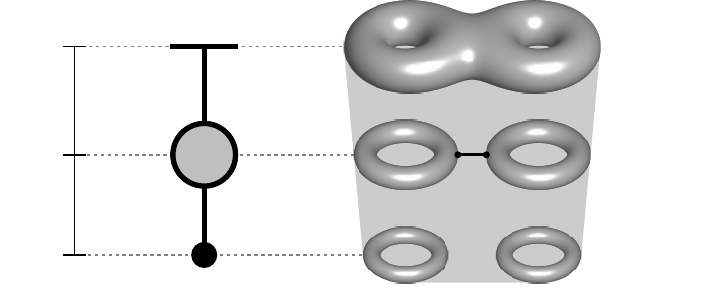}
		\put (0,32.75) {\small{$\phantom{-}1\phantom{-}$}}
		\put (0,17.5) {\small{$\phantom{-}0\phantom{-}$}}
		\put (0,3.25) {\small{$-1\phantom{-}$}}
		\put (88,32.75) {\small{$\Sigma_{1}$}}
		\put (88,17.5) {\small{$\Sigma_{0}$}}
		\put (88,3.25) {\small{$\Sigma_{-1}$}}
         \end{overpic}
     \end{subfigure}
     \begin{subfigure}[b]{0.48\textwidth}
         \centering
         \begin{overpic}[scale=.9]{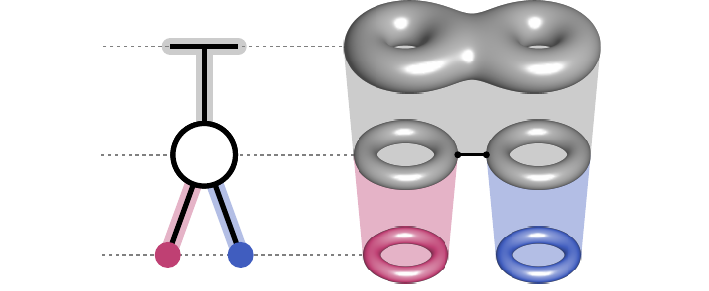}
		\put (4.25,33.25) {\small{grip}}
		\put (4,18) {\small{root}}
		\put (2.5,3.75) {\small{tines}}
		\put (88,32.75) {\small{$\Sigma_{\grip}$}}
		\put (88,17.5) {\small{$\Sigma_{\root}$}}
		\put (88,3.25) {\small{$\Sigma_{\tine_1}\cup\Sigma_{\tine_2}$}}
         \end{overpic}
     \end{subfigure}
     \caption{Sweep-out of a compression body along $[-1,1]$ (left) and along its faithful fork (right).}
     \label{fig:sweep-out-fork}
\end{figure}

\subsection{Generalized Heegaard splittings and fork complexes}

\subparagraph*{Heegaard splittings revisited} Let $\manifold$ be a 3-manifold, and $\{\partial_1\manifold, \partial_2\manifold\}$ be a partition of the components of $\partial\manifold$. A \emph{Heegaard splitting of $(\manifold, \partial_1\manifold, \partial_2\manifold)$} is a triplet $(\compbody_1, \compbody_2, \surface)$, where $\compbody_1$ and $\compbody_2$ are compression bodies with $\compbody_1 \cup \compbody_2 = \manifold$, $\compbody_1 \cap \compbody_1 = \partial_+\compbody_1=\partial_+\compbody_2 = \surface$, $\partial_-\compbody_1 = \partial_1\manifold$, and $\partial_-\compbody_2 = \partial_2\manifold$. The \emph{genus} of the Heegaard splitting $(\compbody_1, \compbody_2,\surface)$ is the genus $g(\surface)$ of the \emph{splitting surface} $\surface$. The \emph{Heegaard genus} $\heeg{\manifold}$ of $\manifold$ is the smallest genus of any Heegaard splitting of $(\manifold, \partial_1\manifold, \partial_2\manifold)$, taken over all partitions $\{\partial_1\manifold, \partial_2\manifold\}$ of $\partial\manifold$.

\begin{figure}[ht]
	\centering
	\begin{overpic}[scale=.7]{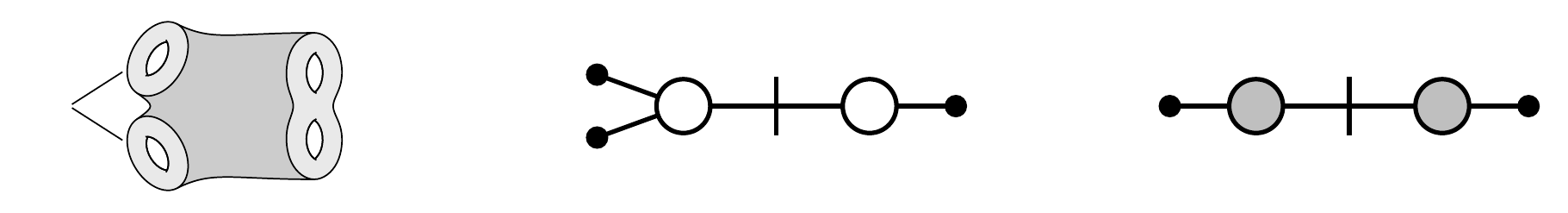}
		\put (-1,6) {\small{$\partial_1\manifold$}}
		\put (13.75,6) {\small{$\manifold$}}
		\put (22.75,6) {\small{$\partial_2\manifold$}}
		\put (31.5,6) {\small{$\partial_-\compbody_1$}}
		\put (38,1) {\scalebox{1.4}{\rotatebox{90}{$\Bigg\{$}}}
		\put (42.75,-1.25) {\small{$\compbody_1$}}
		\put (48.5,9.5) {\small{$\surface$}}
		\put (49.75,1) {\scalebox{1.4}{\rotatebox{90}{$\Bigg\{$}}}
		\put (54.5,-1.25) {\small{$\compbody_2$}}
		\put (63,6) {\small{$\partial_-\compbody_2$}}
		\put (71.75,8.75) {\small{$\partial_-\compbody_1$}}
		\put (74.5,1) {\scalebox{1.4}{\rotatebox{90}{$\Bigg\{$}}}
		\put (79.25,-1.25) {\small{$\compbody_1$}}
		\put (85,9.5) {\small{$\surface$}}
	  	\put (86.25,1) {\scalebox{1.4}{\rotatebox{90}{$\Bigg\{$}}}
		\put (91,-1.25) {\small{$\compbody_2$}}
		\put (95.25,8.75) {\small{$\partial_-\compbody_2$}}		
	\end{overpic}
	\vspace{6pt}
	\caption{Schematic of a $3$-manifold $\manifold$ with a partition of its boundary components~(left). Faithful (center) and non-faithful (right) \emph{fork complexes} representing a Heegaard splitting of $(\manifold, \partial_1\manifold, \partial_2\manifold)$.}
	\label{fig:heegaard-splitting-boundary}
\end{figure}

\begin{proposition}[{\cite[Theorem 2.1.11]{scharlemann2016lecture}}, cf.\ {\cite[Appendix B]{huszar2022pathwidth}}] For any partition $\{\partial_1\manifold, \partial_2\manifold\}$ of the boundary components of $\manifold$, there exists a Heegaard splitting of $(\manifold, \partial_1\manifold, \partial_2\manifold)$.
\label{prop:heegaard}
\end{proposition}

\subparagraph*{Generalized Heegaard splittings} A \emph{generalized Heegaard splitting} $\heegaard$ of a $3$-manifold $\manifold$ consists of
\begin{enumerate*}
	\item a decomposition $\manifold = \bigcup_{i \in I}\manifold_i$ into submanifolds $\manifold_i \subseteq \manifold$ intersecting along closed surfaces (\Cref{fig:manifold-decomp}),
	\item for each $i \in I$ a partition $\{\partial_1\manifold_i, \partial_2\manifold_i\}$ of the components of $\partial\manifold_i$  that together satisfy an \emph{acyclicity condition}: there is an ordering, i.e., a bijection $\ell\colon I \rightarrow \{1,\ldots,|I|\}$, such that, for each $i \in I$ the components of $\partial_1\manifold_i$ (resp.\ $\partial_2\manifold_i$) not belonging to $\partial\manifold$ are only incident to submanifolds $\manifold_j$ with $\ell(j) < \ell(i)$ (resp.\ $\ell(j) > \ell(i)$), and
	\item a choice of a Heegaard splitting $(\compbody^{(i)}_{1}, \compbody^{(i)}_{2}, \surface_i)$ for each $(\manifold_i, \partial_1\manifold_i, \partial_2\manifold_i)$.
\end{enumerate*} Such a choice of \emph{splitting surfaces} $\surface_i$ $(i \in I)$ is said to be \emph{compatible} with $\ell$, see, e.g.,\ \Cref{fig:gen-heegaard}.

\begin{figure}[ht]

	\vspace{24pt}
	
	\begin{subfigure}[b]{0.29\textwidth}
		\centering
		\begin{overpic}[scale=.9]{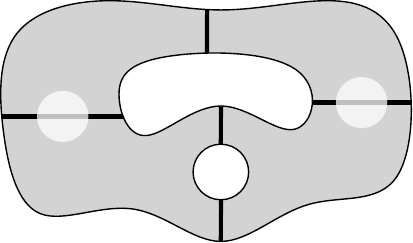}
			\put (15,44) {\small{$\manifold_2$}}
			\put (25,15) {\small{$\manifold_1$}}
			\put (75,46) {\small{$\manifold_4$}}
			\put (73,16) {\small{$\manifold_3$}}
			\put (10,28) {\small{$\altaltsurface_1$}}
			\put (48.5,-10) {\small{$\altaltsurface_2$}}
			\put (45,61) {\small{$\altaltsurface_3$}}
			\put (82.5,31.5) {\small{$\altaltsurface_4$}}
		\end{overpic}
		\vspace{24pt}
		\caption{A decomposition of $\manifold$ into four submanifolds $\manifold_1,\ldots,\manifold_4$ intersecting along (possibly disconnected) closed surfaces $\altaltsurface_i$.}
		\label{fig:manifold-decomp}
	\end{subfigure}
	\hfill
	\begin{subfigure}[b]{0.29\textwidth}
		\centering
		\begin{overpic}[scale=.9]{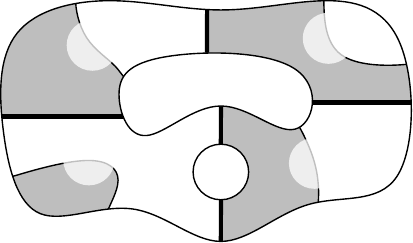}
			\put (18,45.25) {\small{$\surface_2$}}
			\put (17.25,17.25) {\small{$\surface_1$}}
			\put (75.25,47.25) {\small{$\surface_4$}}
			\put (72,16.5) {\small{$\surface_3$}}
			\put (-4,-5) {\footnotesize{$\partial_1\manifold_1 = \emptyset$}}
			\put (-4,-15) {\footnotesize{$\partial_2\manifold_1 = \altaltsurface_1 \cup \altaltsurface_2$}}
			\put (67,-5) {\footnotesize{$\partial_1\manifold_3= \altaltsurface_2$}}
			\put (67,-15) {\footnotesize{$\partial_2\manifold_3 = \altaltsurface_4$}}
			\put (-4,75) {\footnotesize{$\partial_1\manifold_2 = \altaltsurface_1$}}
			\put (-4,65) {\footnotesize{$\partial_2\manifold_2 = \altaltsurface_3$}}
			\put (47,75) {\footnotesize{$\partial_1\manifold_4= \altaltsurface_3 \cup \altaltsurface_4$}}
			\put (47,65) {\footnotesize{$\partial_2\manifold_4 = \emptyset$}}
		\end{overpic}
		\vspace{24pt}
		\caption{An admissible choice of splitting surfaces $\surface_i$ for the $\manifold_i$ that is compatible with the trivial ordering $\ell (i) \mapsto i$.}
		\label{fig:gen-heegaard}
	\end{subfigure}
	\hfill
	\begin{subfigure}[b]{0.29\textwidth}
		\centering
		\begin{overpic}[scale=.9]{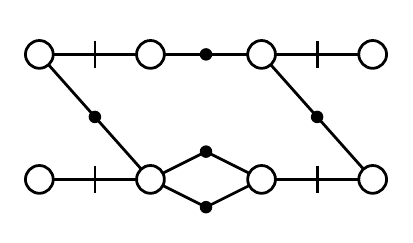}
			\put (19,52) {\small{$\surface_2$}}
			\put (19,2.5) {\small{$\surface_1$}}
			\put (72.75,52) {\small{$\surface_4$}}
			\put (72.75,2.5) {\small{$\surface_3$}}
		\end{overpic}
		\vspace{24pt}
		\caption{The faithful fork complex that represents the generalized Heegaard splitting shown in the center (\Cref{fig:gen-heegaard}).}
		\label{fig:forkcomp}
	\end{subfigure}
	\caption{Schematics of a generalized Heegaard splitting (based on figures from \cite[Section 2.4]{huszar2022pathwidth}).}
	\label{fig:ghs-example}
\end{figure}

Just as compression bodies can be represented by forks, (generalized) Heegaard splittings can be visualized via \emph{fork complexes}, see \Cref{fig:heegaard-splitting-boundary,fig:forkcomp} (cf.\ \cite[Section 5.1]{scharlemann2016lecture} for details).

\subparagraph*{Sweep-outs of 3-manifolds}
A generalized Heegaard splitting $\heegaard$ of a $3$-manifold $\manifold$ induces a \emph{sweep-out} $f\colon\manifold\rightarrow\|\forkcomp\|$ of $\manifold$ along any fork complex $\forkcomp$ that represents $\heegaard$ (here $\|\forkcomp\|$ denotes a drawing, i.e., a \emph{geometric realization} of the abstract fork complex $\forkcomp$) by concatenating the corresponding sweep-outs of the compression bodies that comprise $\heegaard$ (cf.\ \Cref{rem:sweep-out-fork}). We also refer to a sweep-out $f\colon\manifold\rightarrow\|\forkcomp\|$ by the ensemble $\sweepout=\{\sweepout_x : x \in \|\forkcomp\|\}$ of its \emph{level sets}, where $\Sigma_x = f^{-1}(x)$.

\subparagraph*{The width of a generalized Heegaard splitting}

For a generalized Heegaard splitting $\heegaard$, the surfaces $\surface_i$ $(i \in I)$ are also called the \emph{thick levels}, and the lower boundaries $\partial_-\compbody^{(i)}_1$, $\partial_-\compbody^{(i)}_2$ are called the \emph{thin levels} of $\heegaard$. The \emph{width} $\width(\heegaard)$ of $\heegaard$ is the sequence obtained by taking a non-increasing ordering of the multiset $\{g(\surface_i) : i \in I\}$ of the genera of the thick levels.

A generalized Heegaard splitting $\heegaard$ of a 3-manifold $\manifold$ for which $\width(\heegaard)$ is minimal with respect to the lexicographic order ($<$) among all splittings of $\manifold$ is said to be in \emph{thin position}.

\subsection{Weak reductions}
\label{ssec:weak}

A Heegaard splitting $(\compbody_1,\compbody_2,\surface)$ of a connected 3-manifold is said to be \emph{weakly reducible} \cite{casson1987reducing}, if there are essential disks $D_i \subset \compbody_i$ $(i=1,2)$\footnote{The assumption that $D_i$ is essential in the compression body $\compbody_i$ implies that $\partial D_i \subset \partial_+\compbody_i = \surface$ $(i=1,2)$.} with $\partial D_1 \cap \partial D_2 = \emptyset$, see \Cref{fig:weak-reduction-surface}. In this case we also say that the splitting surface $\surface$ is weakly reducible. A generalized Heegaard splitting $\heegaard$ is \emph{weakly reducible}, if at least one of its splitting surfaces is weakly reducible; otherwise $\heegaard$ is called \emph{strongly irreducible}. Every 3-manifold possesses a strongly irreducible generalized Heegaard splitting and this fact can be exploited in various contexts (e.g., in the proof of \Cref{thm:jsj-width}). The usefulness of such splittings is mainly due to the following seminal result.

\begin{theorem}[{\cite[Lemma 5.2.4]{scharlemann2016lecture}}, {\cite[Rule 5]{scharlemann1992thin}}]\label{thm:scharlemann-thompson}
Let $\heegaard$ be a strongly irreducible generalized Heegaard splitting. Then every connected component of every thin level of $\heegaard$ is incompressible.
\end{theorem}

Given a weakly reducible generalized Heegaard splitting $\heegaard$ with a weakly reducible splitting surface\footnote{For future reference we remind the reader that splitting surfaces, also called thick levels, correspond to grips in the faithful fork complex that represents the generalized Heegaard splitting $\heegaard$.}  $\surface$ and essential disks $D_1$ and $D_2$ as above, one can execute a \emph{weak reduction} at $\surface$. This modification amounts to performing particular cut-and-paste operations on $\surface$ guided by $D_1$ and $D_2$, and decomposes each of the two compression bodies adjacent to $\surface$ into a pair of compression bodies. Importantly, this operation results in another generalized Heegaard splitting $\heegaard'$ of the same 3-manifold with $\width(\heegaard') < \width(\heegaard)$.

We illustrate weak reductions via \Cref{ex:weak-reduction} and refer to \cite[Proposition 5.2.3]{scharlemann2016lecture} for further details (notably, Figures 5.8--5.13 therein, but also Lemma 5.2.2, Figures 5.6 and 5.7, and Proposition 5.2.4), including an exhaustive list of instances of weak reductions.\footnote{In the open-access version \cite{scharlemann2005lecture} these are Proposition 4.2.3 and Figures 87--92 (as well as Lemma 4.2.2, Figures 85 and 86, and Lemma 4.2.4).}

\begin{figure}[ht]
	\centering
	\begin{overpic}[scale=.9]{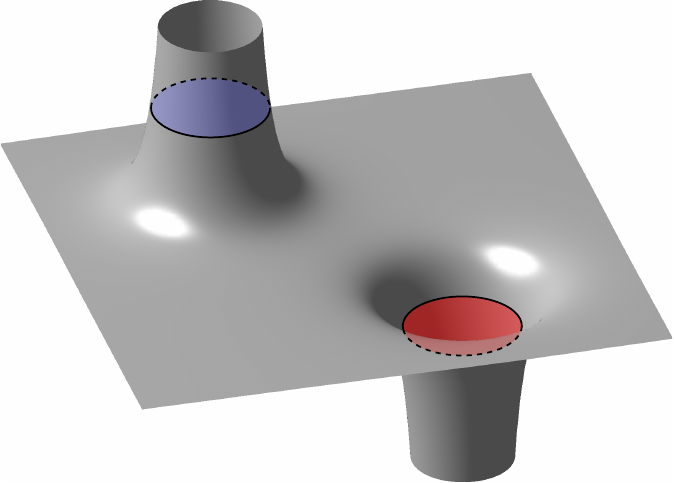}
		\put (28.25,54) {\small{$D_1$}}
		\put (25.25,18.5) {\small{$\surface$}}
		\put (65.5,21.5) {\small{$D_2$}}
	\end{overpic}
	\caption{Local picture of a portion of a weakly reducible splitting surface $\surface$.}
	\label{fig:weak-reduction-surface}
\end{figure}

\section{The Main Result}
\label{sec:main}

In this section we prove Theorem \ref{thm:jsj-width}. The inequalities \eqref{eq:jsj-tw} and \eqref{eq:jsj-pw} are deduced in the same way, thus we only show the proof of \eqref{eq:jsj-tw} in detail, and then explain how it can be adapted to that of \eqref{eq:jsj-pw}. First, we specify what we mean by a ``sufficiently complicated'' JSJ decomposition.

\begin{definition}
\label{def:complicated-jsj}
Given $\delta > 0$, the JSJ decomposition of an irreducible $3$-manifold $\manifold$ is \emph{$\delta$-complicated}, if any incompressible or strongly irreducible Heegaard surface $\surface \subset \manifold$ with genus $g(\surface) \leq \delta$ can be isotoped to be simultaneously disjoint from all the JSJ tori of $\manifold$.
\end{definition}

\begin{proof}[Proof of inequality (\ref{eq:jsj-tw})]
Our goal is to prove that $\tw{\dual(\jsjdecomp)} \leq 18(\tw{\manifold}+1)$, where $\dual(\jsjdecomp)$ is the dual graph of the \emph{$\delta$-complicated} JSJ decomposition of the irreducible 3-manifold $\manifold$. To this end, we set $\delta=18(\tw{\manifold}+1)$ and fix a triangulation $\tri$ of $\manifold$, whose dual graph $\dual(\tri)$ has minimal treewidth, i.e., $\tw{\dual(\tri)} = \tw{\manifold}$. Now, \eqref{eq:jsj-tw} is established in four stages.

\begin{claimproof}[1.\ Setup]
By invoking the construction in \cite[Section 6]{huszar2019treewidth}, from $\tri$ we obtain a generalized Heegaard splitting $\heegaard$ of $\manifold$ together with a sweep-out $f\colon\manifold\rightarrow\|\nfforkcomp\|$ along a non-faithful fork complex $\nfforkcomp$ representing $\heegaard$. By construction, $\nfforkcomp$ is a tree with all of its nodes having degree one or three (\Cref{fig:tree-ghs}), moreover all non-degenerate level surfaces $\sweepout_x = f^{-1}(x)$ have genus bounded above by $18(\tw{\manifold}+1)$.
Let $\fforkcomp$ be the faithful fork complex representing $\heegaard$. Note that $\fforkcomp$ is obtained from $\nfforkcomp$ by replacing every non-faithful fork $\fork \in \nfforkcomp$ with the collection $\mathscr{C}_\fork$ of faithful forks that accurately represents the (possibly disconnected) compression body $\compbody$ corresponding to $\fork$. The inverse operation, i.e., for each $\fork \in \nfforkcomp$ bundling all faithful forks in $\mathscr{C}_\fork$ into $\fork$ (see \Cref{fig:faithful-forks}), induces a projection map $\pi\colon\|\fforkcomp\|\rightarrow\|\nfforkcomp\|$ between the drawings (cf.\ Figures \ref{fig:thm-proof}(i)--(ii)). Note that every fork, tine, or grip in $\nfforkcomp$ corresponds to a collection of the corresponding items in $\fforkcomp$ and so the projection map is well-defined.
\end{claimproof}

\begin{figure}[ht]
	\centering
	\begin{overpic}[scale=.909091]{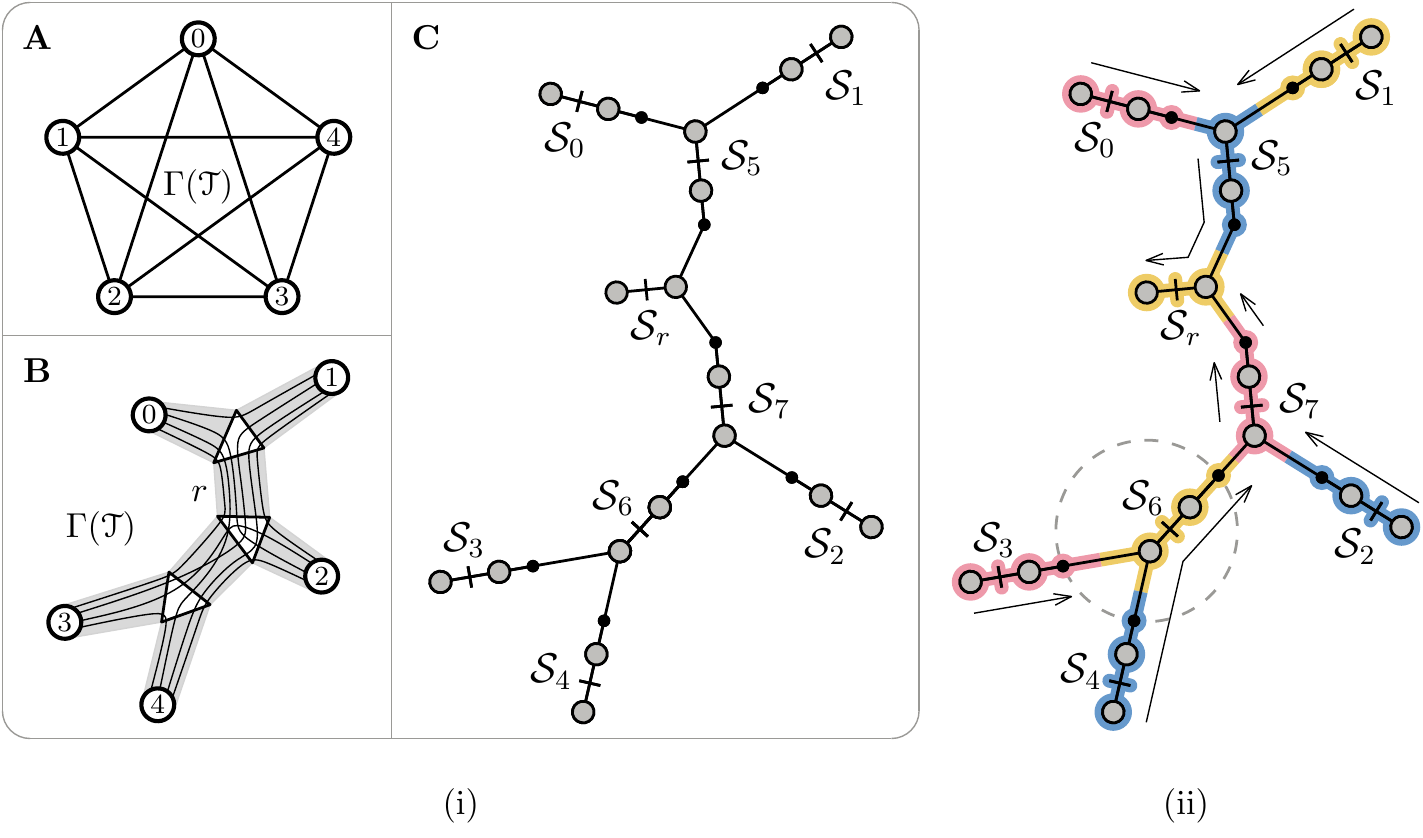}
		\put (54,10) {\Large{$\nfforkcomp$}}
	\end{overpic}
	\caption{(i) \textbf{A.} The dual graph $\dual(\tri)$ of some triangulation $\tri$ of a 3-manifold $\manifold$. \textbf{B.} Low-congestion routing of $\dual(\tri)$ along a host tree with marked root arc $r$. \textbf{C.} Drawing of a non-faithful fork complex $\nfforkcomp$ that represents the generalized Heegaard splitting of $\manifold$ induced by the routing of $\dual(\tri)$. The genera of all (possibly disconnected) thick levels $\surface_i$ is bounded above by $18(\tw{\manifold}+1)$. (ii) Color-coded segmentation of $\|\nfforkcomp\|$ in preparation for the next stage of the proof (cf.\ Figure \ref{fig:thm-proof}).} 
	\label{fig:tree-ghs}
\end{figure}

\begin{claimproof}[2.\ Weak reductions]
In case $\heegaard$ is weakly reducible, we repeatedly perform weak reductions until we obtain a strongly irreducible generalized Heegaard splitting $\heegaard'$ of $\manifold$. Since weak reductions always decrease the width of a generalized Heegaard splitting, this process terminates after finitely many iterations. Throughout, we maintain that the drawings of the associated faithful fork complexes follow $\|\nfforkcomp\|$. Let $\fforkcomp'$ be the faithful fork complex representing the final splitting $\heegaard'$, $f'\colon\manifold\rightarrow\|\fforkcomp'\|$ be the sweep-out of $\manifold$ induced by $\heegaard'$, and $\pi'\colon\|\fforkcomp'\|\rightarrow\|\nfforkcomp\|$ be the associated projection map (\Cref{fig:thm-proof}(iii)). Due to the nature of weak reductions, $18(\tw{\manifold}+1)$ is still an upper bound on the genus of any (possibly disconnected) level surface of $\pi' \circ f' \colon \manifold \rightarrow \|\nfforkcomp\|$. As $\manifold$ is irreducible, we may assume that no component of such a level surface is homeomorphic to a sphere (cf.\ \cite[p.\ 337]{lackenby2006heegaard}). It follows that the number of components of any level set of $\pi' \circ f'$ is at most $18(\tw{\manifold}+1)$.
\end{claimproof}

\begin{claimproof}[3.\ Perturbation and isotopy]
We now apply a level-preserving perturbation $\pi' \leadsto \pi''$ on $\pi'$, after which each thin level of $\heegaard'$ lies in different level sets of $\pi'' \circ f' \colon \manifold \rightarrow \|\nfforkcomp\|$ (\Cref{fig:thm-proof}(iv)). Since $\heegaard'$ is strongly irreducible, its thick levels are strongly irreducible (by definition), and its thin levels are incompressible (by \Cref{thm:scharlemann-thompson}). Moreover, all these surfaces have genera at most $18(\tw{\manifold}+1)$. Hence, as the JSJ decomposition of $\manifold$ is assumed to be $\delta$-complicated with $\delta=18(\tw{\manifold}+1)$, we may isotope all JSJ tori to be disjoint from all the thick and thin levels of $\heegaard'$. Then,  by invoking \cite[Corollary~4.5]{bachman2016heegaard} we can isotope each JSJ torus of $\manifold$ to \emph{coincide} with a component of some thin level of $\heegaard'$ (\Cref{fig:thm-proof}(v)). As a consequence, every compression body of the splitting $\heegaard'$ is contained in a unique JSJ piece of $\manifold$.
\end{claimproof}

\begin{figure}[ht]
	\centering
	\includegraphics[scale=.909091]{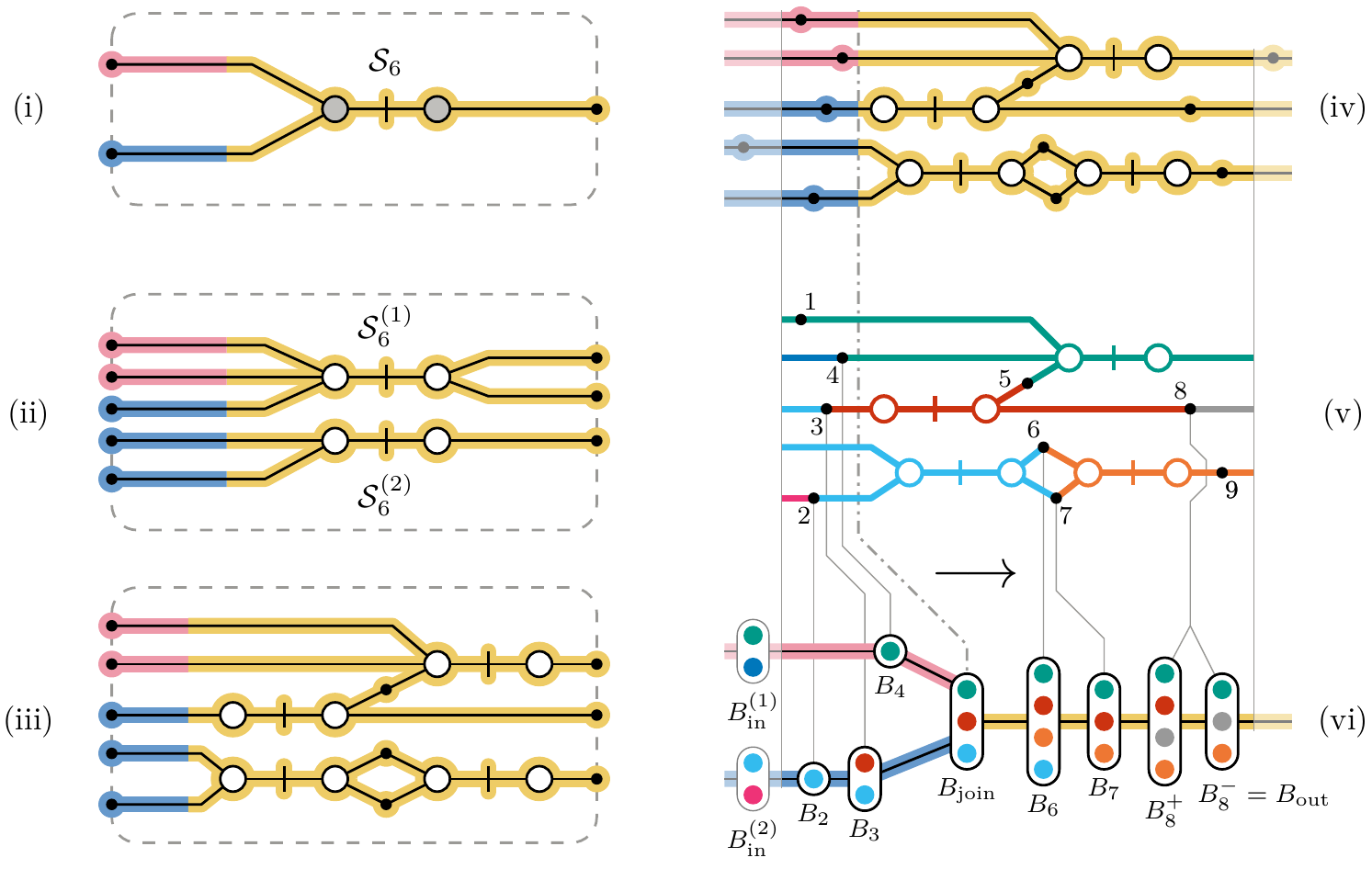}
	\caption{Overview of the proof of inequality (\ref{eq:jsj-tw}). The figure shows the circled area in \Cref{fig:tree-ghs}(ii). Stage 1: Construction of initial fork complex (i) and split into faithful fork complex (ii). Stage 2: Weak reductions (iii), see \cite[Proposition 5.2.3, Figures 5.8--5.13]{scharlemann2016lecture} for a complete list and their effects on the underlying fork complexes. Stage 3: perturbations and isotopy (iv). Stage 4: Construction of tree decomposition (v) and (vi).}
	\label{fig:thm-proof}
\end{figure}

\begin{claimproof}[4.\ The tree decomposition of $\dual(\jsjdecomp)$]
First note that every level set $(\pi'' \circ f')^{-1}(x)$ is incident to at most $18(\tw{\manifold}+1) +1 = 18\tw{\manifold}+19$ JSJ pieces of $\manifold$. The `plus one' appears, because if $(\pi'' \circ f')^{-1}(x)$ contains a JSJ torus, then this torus is incident to two JSJ pieces of $\manifold$. Also note that, because of the perturbation performed in the previous stage, each level set $(\pi'' \circ f')^{-1}(x)$ can contain at most one JSJ torus of $\manifold$.

We now construct a tree decomposition $(\mathscr{X},T)$ of $\dual(\jsjdecomp)$ of width $18(\tw{\manifold}+1)$. Eventually, $T$ will be a subdivision of $\nfforkcomp$ (which is a tree) with nodes corresponding to the bags in $\mathscr{X}$, which we now describe. By \cite[Section 6 (p.\ 86)]{huszar2019treewidth}, each leaf $l$ of $\nfforkcomp$ corresponds to a spine of a handlebody $\hbody_l$. We define a bag that contains the unique node of $\dual(\jsjdecomp)$ associated with the JSJ piece containing $\hbody_l$. As we sweep through $\|\nfforkcomp\|$ (cf.\ the arrows on \Cref{fig:tree-ghs}(2)), whenever we pass through a point $x \in \|\nfforkcomp\|$ such that the level set $(\pi'')^{-1}(x)$ contains a tine of $\|\fforkcomp'\|$, one of four possible events may occur (illustrated in \Cref{fig:thm-proof}(v)--(vi)):
\begin{enumerate*}
	\item A new JSJ piece appears. In this case we take a copy of the previous bag, and add the corresponding node of $\dual(\jsjdecomp)$ into the bag.\label{item:appears}
	\item A JSJ piece disappears. Then we delete the corresponding node of $\dual(\jsjdecomp)$ from a copy of the previous bag.\label{item:disappears}
	\item Both previous kinds of events happen simultaneously. In this case we introduce two new bags. The first to introduce the new JSJ piece, the second to delete the old one.\label{item:both}
	\item If neither a new JSJ piece is introduced, nor an old one is left behind, we do nothing.\label{item:nothing}
\end{enumerate*}
Whenever we arrive at a merging point in the sweep-out (i.e., a degree-three node of $\|\nfforkcomp\|$), we introduce a new bag, which is the union of the two previous bags. Note that the two previous bags do not necessarily need to be disjoint.
\end{claimproof}

It remains to verify that $(\mathscr{X},T)$ is, indeed, a tree decomposition of $\dual(\jsjdecomp)$ of width at most $18\tw{\manifold}+18$. \emph{Node coverage}: Every node of $\dual(\jsjdecomp)$ must be considered at least once, since we sweep through the entirety of $\|\nfforkcomp\|$. \emph{Arc coverage}: we must ensure that all pairs of nodes of $\dual(\jsjdecomp)$ with their JSJ pieces meeting at a tine are contained in some bag. This is always the case for \myref{item:appears}, because a new JSJ piece appears while all other JSJ pieces are still in the bag. It is also the case for \myref{item:disappears} because a JSJ piece disappears but the previous bag contained all the pieces. In \myref{item:both} a new JSJ piece appears and at the same time another JSJ piece disappears. However, in this case we first introduce the new piece, thus making sure that adjacent JSJ pieces always occur in at least one bag. \emph{Sub-tree property}: This follows from the fact that a JSJ piece, once removed from all bags, must be contained in the part of $\|\nfforkcomp\|$ that was already swept. 
Now, every JSJ piece incident to a given level set $(\pi'' \circ f')^{-1}(x)$ of the sweep-out must contribute a positive number to its genus. Hence, it follows that every bag can contain at most  $18\tw{\manifold}+19$ elements (with equality only possible where a tine simultaneously introduces and forgets a JSJ piece). This proves inequality \eqref{eq:jsj-tw}.
\end{proof}

\begin{proof}[Proof of inequality (\ref{eq:jsj-pw})]
We start with the results from \cite[Section 5]{huszar2019treewidth} yielding a fork-complex $\nfforkcomp$ whose underlying space $\|\nfforkcomp\|$ is a path, and the genus of the level sets of the associated sweep-out is bounded above by $4(3\pw{\manifold} + 1)$. Setting $\delta=4(3\pw{\manifold} + 1)$, the remainder of the proof is analogous with the proof of inequality \eqref{eq:jsj-tw}.
\end{proof}

\section{An Algorithmic Construction}
\label{sec:const}

Here we establish \Cref{thm:construction} that paves the way to the applications in \Cref{sec:appl}. In what follows, $\maxdeg$ denotes an arbitrary, but fixed, positive integer. Let $G=(V,E)$ be a graph with $|V|=n$ and maximum degree $\maxdeg$. \Cref{thm:construction} asserts that, in $\poly{n}$ time one can construct a triangulation $\tri_G$ of a closed, irreducible 3-manifold $\manifold_G$, such that the dual graph $\dual(\jsjdecomp)$ of its JSJ decomposition $\jsjdecomp$ equals $G$, moreover, the pathwidth (resp.\ treewidth) of $G$ determines the pathwidth (resp.\ treewidth) of $\manifold_G$ up to a constant factor.

Our proof of \Cref{thm:construction} rests on a synthesis of work by Lackenby \cite{lackenby2017conditionally} and by Bachman, Derby-Talbot and Sedgwick~\cite{bachman2017computing}. In \cite[Section 3]{lackenby2017conditionally} it is shown that the homeomorphism problem for closed 3-manifolds is at least as hard as the graph isomorphism problem. The proof relies on a simple construction that, given a graph $G$, produces a closed, orientable, triangulated 3-manifold whose JSJ decomposition $\jsjdecomp$ satisfies $\dual(\jsjdecomp) = G$. This gives the blueprint for our construction as well. In particular, we use the same building blocks that are described in \cite[p.\ 591]{lackenby2017conditionally}. However, as opposed to Lackenby, we paste together these building blocks via \emph{high-distance torus gluings} akin to the construction presented in \cite[Section~4]{bachman2017computing}. This ensures that we can apply \Cref{thm:jsj-width} for the resulting 3-manifold $\manifold_G$ and deduce the right-hand-side inequalities of \myref{eq:tw-inherit} and \myref{eq:pw-inherit}. The left-hand-side inequalities are shown by inspecting $\tri_G$. We now elaborate on the ingredients of the proof of \Cref{thm:construction}.

\subparagraph*{The building blocks} We first recall the definition of the building blocks from \cite[p.\ 591]{lackenby2017conditionally}. Let $k \in \mathbb{N}$ be a positive integer. Consider the 2-dimensional torus $\torus^2 = \nsphere{1}\times\nsphere{1}$, and let $\torus^2_k$ denote the compact surface obtained from $\torus^2$ by the removal of $k$ pairwise disjoint open disks. We define $\altmanifold(k) = \torus^2_k \times \nsphere{1}$. Note that the boundary $\partial\altmanifold(k)$ of $\altmanifold(k)$ consists of $k$ tori, which will be the gluing sites. See \Cref{fig:building-block} for the example of $\altmanifold(3)$.

We choose a triangulation $\tri(k)$ for $\altmanifold(k)$ that induces the minimal 2-triangle triangulation of the torus (cf.\ \Cref{fig:torus-trg}) at each boundary component. Note that $\tri(k)$ can be constructed from $O(k)$ tetrahedra. An explicit description of $\tri(k)$ is given in \Cref{app:trg}.

\bigskip

\begin{nolinenumbers}
\noindent
\begin{minipage}{0.425\textwidth}
	\centering
	\begin{overpic}[scale=0.75]{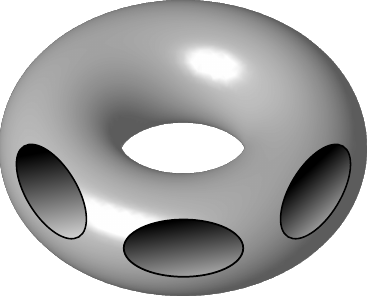}
		\put (-47,37) {\large{$\altmanifold(3) =$}}
		\put (106,37) {\large{$\times ~\nsphere{1}$}}
	\end{overpic}
	\captionof{figure}{Illustration of $\altmanifold(3)$.}
	\label{fig:building-block}
\end{minipage}\hfill%
\begin{minipage}{0.55\textwidth}
	\centering
	\begin{overpic}[scale=0.75]{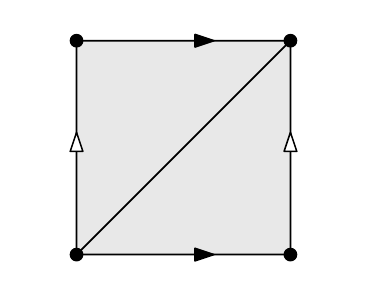}
	\end{overpic}
	\captionof{figure}{\begin{nolinenumbers}Minimal triangulation of the torus $\torus^2$.\end{nolinenumbers}}
	\label{fig:torus-trg}
\end{minipage}
\end{nolinenumbers}

\subparagraph*{Overview of the construction of $\boldsymbol{\manifold_G}$} We now give a high-level overview of constructing $\manifold_G$ from the above building blocks. Given a graph $G=(V,E)$, for each node $v \in V$ we pick a block $\altmanifold_v\cong\altmanifold(d_v)$, where $d_v$ denotes the degree of $v$. Next, for each arc $e = \{u,v\} \in E$, we pick a homeomorphism $\phi_e\colon\torus^2\rightarrow\torus^2$ of sufficiently high \emph{distance} (we discuss this notion below) and use it to glue together a boundary torus of $\altmanifold_u$ with one of $\altmanifold_v$. After performing all of these gluings, we readily obtain the 3-manifold $\manifold_G$ (cf.\ \Cref{fig:schematic-construction}(i)--(ii)). 

\begin{claim}[name={Based on {\cite[p.\ 591]{lackenby2017conditionally}}}, label=claim:jsj]
The JSJ decomposition $\jsjdecomp$ of $\manifold_G$ satisfies $\dual(\jsjdecomp) = G$.
\end{claim}

For a proof of \Cref{claim:jsj} we refer to \Cref{app:jsj}.

\begin{figure}[ht]
\centering
\begin{overpic}[scale=\MyFigScaleHuge]{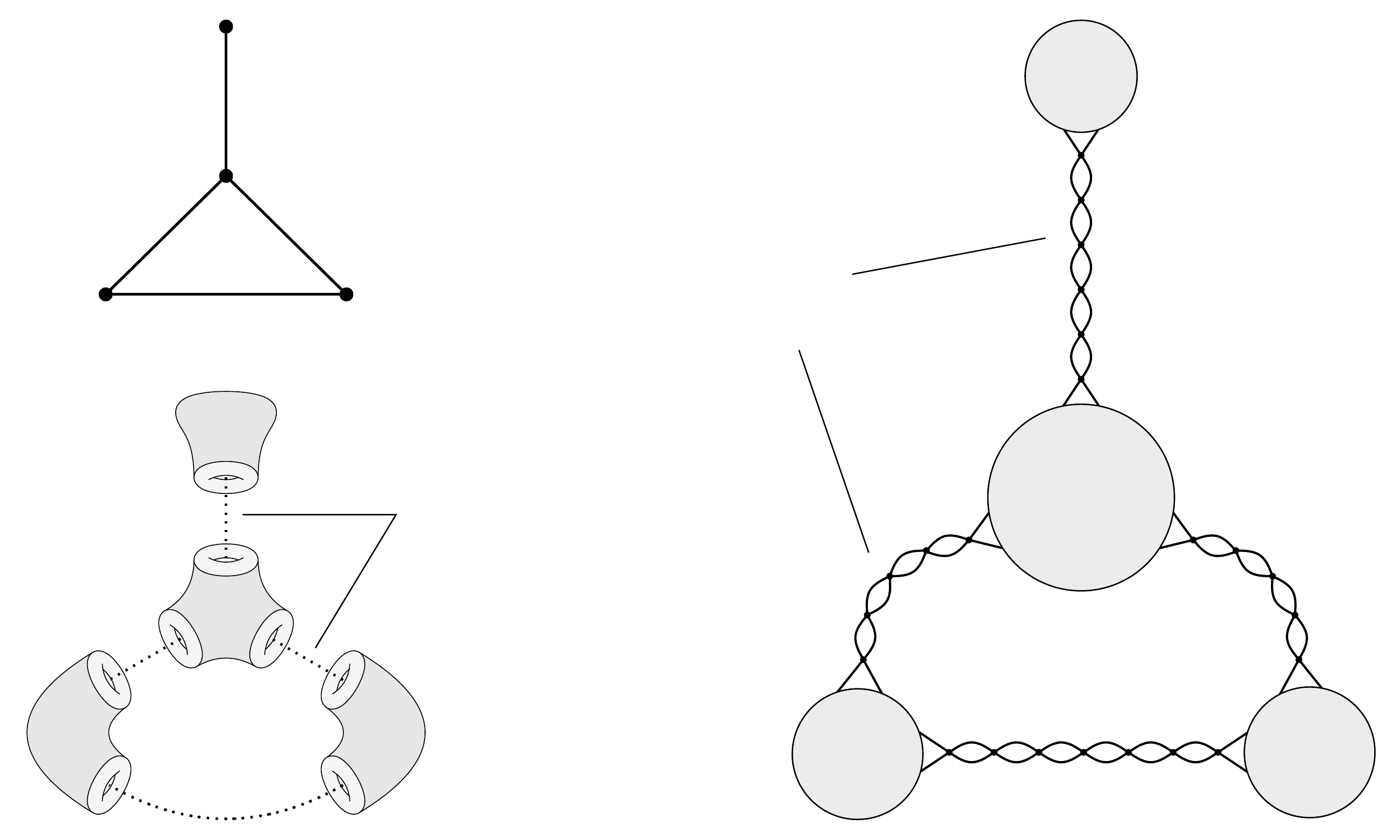}
	\put (30,24.5) {\small{high-distance}}
	\put (30,21) {\small{torus gluings}}
	\put (51,40.5) {\small{layered}}
	\put (51,37) {\small{triangulations}}
	\put (14.5,-4) {\large{$\manifold_G$}}
	\put (74,-4) {\large{$\dual(\tri_G)$}}
	\put (-1.5,19) {\small{$\altmanifold(3) \cong \altmanifold_v$}}
	\put (-1.5,28) {\small{$\altmanifold(1) \cong \altmanifold_u$}}
	\put (-1,1) {\small{$\altmanifold_w$}}
	\put (28.75,1) {\small{$\altmanifold_z \cong \altmanifold(2)$}}
	\put (74,54) {\scalebox{0.7}{{$\dual(\tri(1))$}}}
	\put (73,23.75) {\small{$\dual(\tri(3))$}}
	\put (57.5,5.45) {\scalebox{0.8}{{$\dual(\tri(2))$}}}
	\put (89.8,5.45) {\scalebox{0.8}{{$\dual(\tri(2))$}}}
	\put (20,52) {\large{$G$}}
	\put (12.5,58) {\small{$u$}}
	\put (12.5,48) {\small{$v$}}
	\put (3.5,38.5) {\small{$w$}}
	\put (27,38.5) {\small{$z$}}
	\put (-8,50) {\small{(i)}}
	\put (-8,15) {\small{(ii)}}
	\put (92,50) {\small{(iii)}}
\end{overpic}

\vspace{21pt}

\caption{Schematic overview of the construction underlying \Cref{thm:construction}.}
\label{fig:schematic-construction}
\end{figure}

\subparagraph*{High-distance torus gluings} As already mentioned, to ensure that we can apply \Cref{thm:jsj-width} to $\manifold_G$, we use torus homeomorphisms of ``sufficiently high distance'' to glue the building blocks together. This notion of \emph{distance}, which is defined through the \emph{Farey distance}, is somewhat technical. Thus, we refer to \cite[Section 4.1 and Appendix]{bachman2017computing} for details, and rather recall a crucial result that makes the usefulness of distance in the current context apparent.

\begin{theorem}[{\cite[Appendix]{bachman2017computing}, \cite[Theorem 5.4]{bachman2013stabilizing}}\protect\footnote{The notation and the statement of \Cref{thm:distance} have been adapted to match the present context.}] There exists a computable constant $K$, depending only on the homeomorphism types of the blocks, so that if any set of blocks are glued with maps of distance at least $K \delta$ along their torus boundary components to form a closed $3$-manifold $\manifold$, then the JSJ decomposition of $\manifold$ is $\delta$-complicated (cf.\ \Cref{def:complicated-jsj}).
\label{thm:distance}
\end{theorem}

\subparagraph*{Triangulating the gluing maps} We have already discussed that the block $\altmanifold(k)$ admits a triangulation $\tri(k)$ with $O(k)$ tetrahedra, where $\tri(k)$ induces a minimal, 1-vertex triangulation at each torus boundary of $\altmanifold(k)$. It is shown in \cite[Section 4.2]{bachman2017computing} that the gluings beading these blocks together can be realized as \emph{layered triangulations} \cite{jaco2006layered}. These triangulations manifest as ``daisy chains'' in the dual graph $\dual(\tri_G)$ of the final triangulation $\tri_G$, see \Cref{fig:schematic-construction}(iii).

\begin{lemma}[{\cite[Lemma 4.6]{bachman2017computing}}]
\label{lem:layering}
There exist torus gluings with distance at least $D$, that can be realized as layered triangulations using $2D$ tetrahedra.
\end{lemma}

\begin{claim}
\label{claim:lhs-inequalities}
There exist universal constants $c,c' > 0$ such that
\vspace{-.5\topsep}
\begin{align*}
	(c / \maxdeg) \tw{\manifold_G} \leq \tw{G} \quad \text{and} \quad (c' / \maxdeg) \pw{\manifold_G} \leq \pw{G}.
\end{align*}
\end{claim}

\begin{claimproof}
Since every node of $G$ has degree at most $\maxdeg$, the construction of $\manifold_G$ only uses building blocks homeomorphic to $\altmanifold(1),\ldots,\altmanifold(\maxdeg)$. Hence, for each $v \in V$ the triangulation $\tri(d_v)$ of the block $\altmanifold_v$ contains $O(\maxdeg)$ tetrahedra. This, together with the above discussion on triangulating the gluing maps implies that (upon ignoring multi-arcs and loop arcs, which are anyway not ``sensed'' by treewidth or pathwidth), the dual graph $\dual(\tri_G)$ is obtained from $G$ by \begin{enumerate*}
	\item replacing each node $v \in V$ with a copy of the graph $\dual(\tri(d_v))$ that contains $O(\Delta)$ nodes, and by
	\item possibly subdividing each arc $e \in E$ several times.
\end{enumerate*}

Now, the first operation increases the treewidth (resp.\ pathwidth) at most by a factor of $O(\maxdeg)$, while the arc-subdivisions keep these parameters basically the same, cf.\ \Cref{lem:subdiv-pw}. Hence $\tw{\dual(\tri_G)} \leq O(\maxdeg\tw{G})$ and $\pw{\dual(\tri_G)} \leq O(\maxdeg\pw{G})$, and the claim follows.
\end{claimproof}

\begin{lemma}[Folklore, cf.\ {\cite[Lemma A. 1]{belmonte2022parameterized}}]
\label{lem:subdiv-pw} Let $G=(V,E)$ be a graph. If $G'$  is a graph obtained from $G$ by subdividing a set $F \subseteq E$ of arcs an arbitrary number of times. Then
\begin{align*}
	\pw{G'} \leq \pw{G}+2 \qquad \text{and} \qquad \tw{G'} \leq \max\{\tw{G},3\}.
\end{align*}
 \end{lemma}

\begin{proof}[Finishing the proof of \Cref{thm:construction}]
We have already shown \myref{thm:construction-jsj} (\Cref{claim:jsj}) and the left-hand-sides of the inequalities  \myref{eq:tw-inherit} and \myref{eq:pw-inherit} (\Cref{claim:lhs-inequalities}). To prove the remaining parts of \Cref{thm:construction}, let $\delta = \max\{18(\tw{G} + 1),4(3\pw{G} + 1)\} = O(\pw{G})$. By \Cref{thm:distance}, there is a computable constant $K_\maxdeg$ depending only on $\altmanifold(1),\ldots,\altmanifold(\maxdeg)$ and hence only on $\maxdeg$, so that if we glue together the blocks via maps of distance at least $K_\maxdeg\delta$, then the JSJ decomposition of $\manifold_G$ is $\delta$-complicated. By \Cref{lem:layering}, each such gluing map can be realized as a layered triangulation consisting of $2K_\maxdeg\delta$ tetrahedra. Since $G$ has at most $\maxdeg n/2$ arcs, these layered triangulations contain at most $2K_\maxdeg\delta\maxdeg n/2 = \maxdeg K_\maxdeg\delta n = O(\maxdeg K_\maxdeg\pw{G} \cdot n) = O_\maxdeg(\pw{G} \cdot n)$ tetrahedra altogether. Since the triangulated blocks $\tri(d_v)$ contain $O(\maxdeg\cdot n)$ tetrahedra in total, the triangulation $\tri_G$ of the manifold $\manifold_G$ can be built from $O_\maxdeg(\pw{G} \cdot n) \leq O_\maxdeg(n^2)$ tetrahedra. Last, as it follows from \cite[Section 4]{bachman2017computing}, the construction can be executed in quadratic time.
\end{proof}

\section{Applications}
\label{sec:appl}

\begin{corollary}
\label{cor:unbounded-pathwidth}
There exist $3$-manifolds $(\manifold_h)_{h\in \mathbb{N}}$ with $\tw{\manifold_h} \leq 2$ and $\pw{\manifold_h} \overset{h\to \infty}{\longrightarrow} \infty$. 
\end{corollary}

\begin{corollary}
\label{cor:unbounded-treewidth}
There exist Haken $3$-manifolds $(\altmanifold_k)_{k\in \mathbb{N}}$ with $\tw{\altmanifold_k} \overset{k\to \infty}{\longrightarrow} \infty$.
\end{corollary}

\begin{proof}[Proof of {\Cref{cor:unbounded-pathwidth,cor:unbounded-treewidth}}]
Using Theorem \ref{thm:construction}, the construction of bounded-treewidth (Haken) 3-manifolds with arbitrarily large pathwidth follows by taking $\manifold_h = \manifold_{T_h}$ (with the notation of \Cref{thm:construction}), where $T_h$ is the \emph{complete binary tree of height $h$}. The construction of Haken 3-manifolds of arbitrary large treewidth is deduced by setting $\altmanifold_k = \manifold_{\operatorname{Grid}(k)}$, where $\operatorname{Grid}(k)$ denotes the $k \times k$ grid graph. See \Cref{fig:pw-tw}. Obtained as JSJ decompositions, where the JSJ tori are two-sided, incompressible surfaces, all of these manifolds are Haken.
\end{proof}

\begin{figure}[ht]
\centering
\begin{overpic}[scale=\MyFigScale]{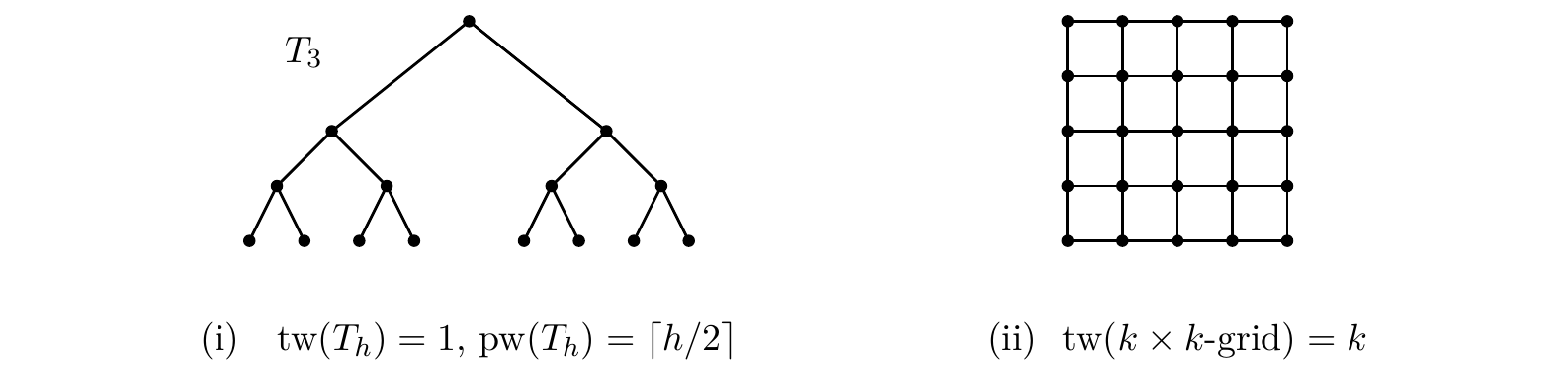}
	\put (56.5,19.25) {\footnotesize{$\operatorname{Grid}(5)$}}
\end{overpic}
\caption{(i) The complete binary tree $T_h$ of height $h$ has pathwidth $\lceil h/2 \rceil$, cf.\ \cite[Theorem 67]{bodlaender1998partial}. (ii) The $k \times k$ grid graph is known to have pathwidth and treewidth both equal to $k$.}
\label{fig:pw-tw}
\end{figure}

\begin{corollary}\label{cor:hardness-treewidth}
	Approximating the treewidth (resp.\ pathwidth) of closed, orientable $3$-manifolds up to a constant factor is at least as hard as giving a constant-factor approximation of the treewidth (resp.\ pathwidth) of bounded-degree graphs. 
\end{corollary}

\begin{proof}
The argument for treewidth and pathwidth is the same. Given a graph $G$ with maximum vertex degree $\maxdeg$, we use the polynomial-time procedure from \Cref{thm:construction} to build a $3$-manifold $\manifold$ with $\tw{\manifold}$ within a constant factor of $\tw{G}$. An oracle for a constant-factor approximation of $\tw{\manifold}$ hence gives us a constant-factor approximation of $\tw{G}$ as well.
\end{proof}

\begin{remark}
Computing a constant-factor approximation of treewidth (resp.\ pathwidth) for arbitrary graphs is known to be conditionally {\bf NP}-hard under the Small Set Expansion Hypothesis \cite{raghavendra2010graph,wu2014inapproximability, yamazaki2018inapproximability}. For proving \Cref{cor:hardness-treewidth}, however, we rely on the assumption that the graph has bounded degree. Thus the conditional hardness of approximating the treewidth (resp.\ pathwidth) of a $3$-manifold does not directly follow.
Establishing such a hardness result would add to the growing, but still relatively short list of algorithmic problems in low-dimensional topology that are known to be (conditionally) hard, cf.\ \cite{agol2006computational, bachman2017computing, mesmay2021unbearable, koenig2021nphard, lackenby2017conditionally}.
\end{remark}

\section{Discussion and Open Problems}
\label{sec:discussion}

We have demonstrated that $3$-manifolds with JSJ decompositions with dual graphs of large treewidth (resp.\ pathwidth) and ``sufficiently complicated'' gluing maps cannot admit triangulations of low treewidth (resp.\ pathwidth). This provides a technique to construct a wealth of families of $3$-manifolds with unbounded tree- or pathwidth that hopefully will prove to be useful for future research in the field.

One obvious limitation of our construction is a seemingly heavy restriction on the JSJ gluing maps in order to deduce a connection between the treewidth of a $3$-manifold and that of the dual graph of its JSJ decomposition. Hence, a natural question to ask is, how much this restriction on gluing maps may be relaxed while still allowing meaningful structural results about treewidth. In particular, we have the following question.

\begin{question}
  Given a $3$-manifold $\manifold$ with JSJ decomposition $\jsjdecomp$ and no restrictions on its JSJ gluings. Is there a lower bound for $\tw{\manifold}$ in terms of $\tw{\dual(\jsjdecomp)}$?  
\end{question}

Note that the assumption that we are considering the JSJ decomposition of $\manifold$ is necessary: Consider a graph $G=(V,E)$ and a collection of $3$-manifolds $\{ \manifold_v\}_{v \in V}$, where $\manifold_v$ has $\deg (v)$ torus boundary components. Assume that we glue the manifolds $\manifold_v$ along the arcs of $G$ to obtain a closed $3$-manifold $\manifold$. Without restrictions on how these pieces are glued together, this cannot result in a lower bound $\tw{\manifold}$ in terms of $\tw{G}$: we can construct Seifert fibered spaces $\manifold$ in this way, even if $G=(V,E)$ is the complete graph with $|V|$ arbitrarily large. At the same time, Seifert fibered spaces have constant treewidth, see \cite{huszar2019manifold}.

\begin{question}
What is the complexity of computing the treewidth of a $3$-manifold? 
\end{question}

We believe that this should be at least as hard as computing the treewidth of a graph. 

\bibliography{references}

\begin{thebibliography}{10}

\bibitem{agol2003small}
I.~Agol.
\newblock Small 3-manifolds of large genus.
\newblock {\em Geom. Dedicata}, 102:53--64, 2003.
\newblock \href {https://doi.org/10.1023/B:GEOM.0000006584.85248.c5}
  {\path{doi:10.1023/B:GEOM.0000006584.85248.c5}}.

\bibitem{agol2006computational}
I.~Agol, J.~Hass, and W.~Thurston.
\newblock The computational complexity of knot genus and spanning area.
\newblock {\em Trans. Am. Math. Soc.}, 358(9):3821--3850, 2006.
\newblock \href {https://doi.org/10.1090/S0002-9947-05-03919-X}
  {\path{doi:10.1090/S0002-9947-05-03919-X}}.

\bibitem{armstrong1983basic}
M.~A. Armstrong.
\newblock {\em Basic topology}.
\newblock Undergrad. Texts Math. Springer-Verlag, New York-Berlin, 1983.
\newblock Corrected reprint of the 1979 original.
\newblock \href {https://doi.org/10.1007/978-1-4757-1793-8}
  {\path{doi:10.1007/978-1-4757-1793-8}}.

\bibitem{bachman2013stabilizing}
D.~Bachman.
\newblock Stabilizing and destabilizing {H}eegaard splittings of sufficiently
  complicated 3-manifolds.
\newblock {\em Math. Ann.}, 355(2):697--728, 2013.
\newblock \href {https://doi.org/10.1007/s00208-012-0802-4}
  {\path{doi:10.1007/s00208-012-0802-4}}.

\bibitem{bachman2016heegaard}
D.~Bachman, R.~Derby-Talbot, and E.~Sedgwick.
\newblock Heegaard structure respects complicated {JSJ} decompositions.
\newblock {\em Math. Ann.}, 365(3-4):1137--1154, 2016.
\newblock \href {https://doi.org/10.1007/s00208-015-1314-9}
  {\path{doi:10.1007/s00208-015-1314-9}}.

\bibitem{bachman2017computing}
D.~Bachman, R.~Derby-Talbot, and E.~Sedgwick.
\newblock Computing {H}eegaard genus is {NP}-hard.
\newblock In {\em {A Journey Through Discrete Mathematics: A Tribute to
  Ji\v{r}\'i Matou\v{s}ek}}, pages 59--87. Springer, Cham, 2017.
\newblock \href {https://doi.org/10.1007/978-3-319-44479-6_3}
  {\path{doi:10.1007/978-3-319-44479-6_3}}.

\bibitem{bachman2006sweepouts}
D.~Bachman, S.~Schleimer, and E.~Sedgwick.
\newblock Sweepouts of amalgamated 3-manifolds.
\newblock {\em Algebr. Geom. Topol.}, 6:171--194, 2006.
\newblock \href {https://doi.org/10.2140/agt.2006.6.171}
  {\path{doi:10.2140/agt.2006.6.171}}.

\bibitem{bagchi2016efficient}
B.~Bagchi, B.~A. Burton, B.~Datta, N.~Singh, and J.~Spreer.
\newblock Efficient algorithms to decide tightness.
\newblock In {\em 32nd {I}nt. {S}ymp. {C}omput. {G}eom. ({SoCG} 2016)},
  volume~51 of {\em LIPIcs. Leibniz Int. Proc. Inf.}, pages 12:1--12:15.
  Schloss Dagstuhl--Leibniz-Zent. Inf., 2016.
\newblock \href {https://doi.org/10.4230/LIPIcs.SoCG.2016.12}
  {\path{doi:10.4230/LIPIcs.SoCG.2016.12}}.

\bibitem{belmonte2022parameterized}
R.~Belmonte, T.~Hanaka, M.~Kanzaki, M.~Kiyomi, Y.~Kobayashi, Y.~Kobayashi,
  M.~Lampis, H.~Ono, and Y.~Otachi.
\newblock Parameterized complexity of {$(A,\ell )$}-path packing.
\newblock {\em Algorithmica}, 84(4):871--895, 2022.
\newblock \href {https://doi.org/10.1007/s00453-021-00875-y}
  {\path{doi:10.1007/s00453-021-00875-y}}.

\bibitem{bing1959alternative}
R.~H. Bing.
\newblock An alternative proof that {$3$}-manifolds can be triangulated.
\newblock {\em Ann. Math. (2)}, 69:37--65, 1959.
\newblock \href {https://doi.org/10.2307/1970092} {\path{doi:10.2307/1970092}}.

\bibitem{bodlaender1998partial}
H.~L. Bodlaender.
\newblock A partial \emph{k}-arboretum of graphs with bounded treewidth.
\newblock {\em Theor. Comput. Sci.}, 209(1--2):1--45, 1998.
\newblock \href {https://doi.org/10.1016/S0304-3975(97)00228-4}
  {\path{doi:10.1016/S0304-3975(97)00228-4}}.

\bibitem{burton2013regina}
B.~A. Burton.
\newblock Computational topology with {R}egina: algorithms, heuristics and
  implementations.
\newblock In {\em Geometry and Topology Down Under}, volume 597 of {\em
  Contemp. Math.}, pages 195--224. Am. Math. Soc., Providence, RI, 2013.
\newblock \href {https://doi.org/10.1090/conm/597/11877}
  {\path{doi:10.1090/conm/597/11877}}.

\bibitem{Regina}
B.~A. Burton, R.~Budney, W.~Pettersson, et~al.
\newblock Regina: Software for low-dimensional topology, 1999--2022.
\newblock Version 7.2.
\newblock URL: \url{https://regina-normal.github.io}.

\bibitem{burton2017courcelle}
B.~A. Burton and R.~G. Downey.
\newblock Courcelle's theorem for triangulations.
\newblock {\em J. Comb. Theory, Ser. {A}}, 146:264--294, 2017.
\newblock \href {https://doi.org/10.1016/j.jcta.2016.10.001}
  {\path{doi:10.1016/j.jcta.2016.10.001}}.

\bibitem{burton2016parameterized}
B.~A. Burton, T.~Lewiner, J.~Paix{\~{a}}o, and J.~Spreer.
\newblock Parameterized complexity of discrete {M}orse theory.
\newblock {\em {ACM} Trans. Math. Softw.}, 42(1):6:1--6:24, 2016.
\newblock \href {https://doi.org/10.1145/2738034} {\path{doi:10.1145/2738034}}.

\bibitem{burton2018algorithms}
B.~A. Burton, C.~Maria, and J.~Spreer.
\newblock Algorithms and complexity for {Turaev--Viro} invariants.
\newblock {\em J. Appl. Comput. Topol.}, 2(1--2):33--53, 2018.
\newblock \href {https://doi.org/10.1007/s41468-018-0016-2}
  {\path{doi:10.1007/s41468-018-0016-2}}.

\bibitem{pettersson2014fixed}
B.~A. Burton and W.~Pettersson.
\newblock Fixed parameter tractable algorithms in combinatorial topology.
\newblock In {\em Proc. 20th Int. Conf. Comput. Comb. ({COCOON} 2014)}, pages
  300--311, 2014.
\newblock \href {https://doi.org/10.1007/978-3-319-08783-2_26}
  {\path{doi:10.1007/978-3-319-08783-2_26}}.

\bibitem{burton2013complexity}
B.~A. Burton and J.~Spreer.
\newblock The complexity of detecting taut angle structures on triangulations.
\newblock In {\em Proc. 24th Annu. {ACM-SIAM} Symp. Discrete Algorithms ({SODA}
  2013)}, pages 168--183, 2013.
\newblock \href {https://doi.org/10.1137/1.9781611973105.13}
  {\path{doi:10.1137/1.9781611973105.13}}.

\bibitem{casson1987reducing}
A.~J. Casson and C.~McA. Gordon.
\newblock Reducing {H}eegaard splittings.
\newblock {\em Topology Appl.}, 27(3):275--283, 1987.
\newblock \href {https://doi.org/10.1016/0166-8641(87)90092-7}
  {\path{doi:10.1016/0166-8641(87)90092-7}}.

\bibitem{coxeter1956regular}
H.~S.~M. Coxeter.
\newblock Regular honeycombs in hyperbolic space.
\newblock In {\em Proc. Int. Congr. Math., Amsterdam, The Netherlands,
  September 2--9, 1954}, volume~3, pages 155--169. Erven P. Noordhoff;
  North-Holland, 1956.
\newblock Available at
  \url{https://www.mathunion.org/fileadmin/ICM/Proceedings/ICM1954.3/ICM1954.3.ocr.pdf}
  (accessed: March 12, 2023).

\bibitem{mesmay2021unbearable}
A.~de~Mesmay, Y.~Rieck, E.~Sedgwick, and M.~Tancer.
\newblock The unbearable hardness of unknotting.
\newblock {\em Adv. Math.}, 381:Paper No. 107648, 36, 2021.
\newblock \href {https://doi.org/10.1016/j.aim.2021.107648}
  {\path{doi:10.1016/j.aim.2021.107648}}.

\bibitem{saint-gervais2016uniformisation}
Henri~Paul de~Saint-Gervais.
\newblock {\em Uniformization of {Riemann} surfaces. {Revisiting} a
  hundred-year-old theorem. {Translated} from the {French} by {R.} {G}.
  {Burns}}.
\newblock Herit. Eur. Math. Eur. Math. Soc. (EMS), Z\"{u}rich, 2016.
\newblock \href {https://doi.org/10.4171/145} {\path{doi:10.4171/145}}.

\bibitem{hatcher2007notes}
A.~Hatcher.
\newblock Notes on basic 3-manifold topology.
\newblock Available at \url{https://pi.math.cornell.edu/~hatcher/3M/3Mfds.pdf}
  (accessed: February 18, 2023).

\bibitem{heegaard1916analysis}
P.~Heegaard.
\newblock Sur l'"{A}nalysis situs".
\newblock {\em Bull. Soc. Math. France}, 44:161--242, 1916.
\newblock \href {https://doi.org/10.24033/bsmf.968}
  {\path{doi:10.24033/bsmf.968}}.

\bibitem{huszar2020combinatorial}
K.~Husz\'ar.
\newblock {\em Combinatorial width parameters for {$3$}-dimensonal manifolds}.
\newblock PhD thesis, IST Austria, June 2020.
\newblock \href {https://doi.org/10.15479/AT:ISTA:8032}
  {\path{doi:10.15479/AT:ISTA:8032}}.

\bibitem{huszar2022pathwidth}
K.~Husz\'ar.
\newblock On the pathwidth of hyperbolic 3-manifolds.
\newblock {\em Comput. Geom. Topol.}, 1(1):1--19, 2022.
\newblock \href {https://doi.org/10.57717/cgt.v1i1.4}
  {\path{doi:10.57717/cgt.v1i1.4}}.

\bibitem{huszar2019manifold}
K.~Husz\'ar and J.~Spreer.
\newblock 3-{M}anifold triangulations with small treewidth.
\newblock In {\em 35th {I}nt. {S}ymp. {C}omput. {G}eom. ({SoCG} 2019)}, volume
  129 of {\em LIPIcs. Leibniz Int. Proc. Inf.}, pages 44:1--44:20. Schloss
  Dagstuhl--Leibniz-Zent. Inf., 2019.
\newblock \href {https://doi.org/10.4230/LIPIcs.SoCG.2019.44}
  {\path{doi:10.4230/LIPIcs.SoCG.2019.44}}.

\bibitem{huszar2019treewidth}
K.~Husz\'ar, J.~Spreer, and U.~Wagner.
\newblock On the treewidth of triangulated 3-manifolds.
\newblock {\em J. Comput. Geom.}, 10(2):70--98, 2019.
\newblock \href {https://doi.org/10.20382/jogc.v10i2a5}
  {\path{doi:10.20382/jogc.v10i2a5}}.

\bibitem{jaco2006layered}
W.~Jaco and J.~H. Rubinstein.
\newblock Layered-triangulations of 3-manifolds, 2006.
\newblock 97 pages, 32 figures.
\newblock \href {http://arxiv.org/abs/math/0603601}
  {\path{arXiv:math/0603601}}.

\bibitem{jaco1978decomposition}
W.~Jaco and P.~B. Shalen.
\newblock A new decomposition theorem for irreducible sufficiently-large
  $3$-manifolds.
\newblock In {\em Algebraic and Geometric Topology}, volume 32, part 2 of {\em
  Proc. Sympos. Pure Math.}, pages 71--84. Am. Math. Soc., Providence, RI,
  1978.
\newblock \href {https://doi.org/10.1090/pspum/032.2/520524}
  {\path{doi:10.1090/pspum/032.2/520524}}.

\bibitem{jaco1979seifert}
W.~H. Jaco and P.~B. Shalen.
\newblock Seifert fibered spaces in {$3$}-manifolds.
\newblock {\em Mem. Am. Math. Soc.}, 21(220):viii+192, 1979.
\newblock \href {https://doi.org/10.1090/memo/0220}
  {\path{doi:10.1090/memo/0220}}.

\bibitem{johannson1979homotopy}
K.~Johannson.
\newblock {\em Homotopy equivalences of {$3$}-manifolds with boundaries},
  volume 761 of {\em Lect. Notes Math.}
\newblock Springer, Berlin, 1979.
\newblock \href {https://doi.org/10.1007/BFb0085406}
  {\path{doi:10.1007/BFb0085406}}.

\bibitem{koenig2021nphard}
D.~Koenig and A.~Tsvietkova.
\newblock N{P}-hard problems naturally arising in knot theory.
\newblock {\em Trans. Am. Math. Soc. Ser. B}, 8:420--441, 2021.
\newblock \href {https://doi.org/10.1090/btran/71}
  {\path{doi:10.1090/btran/71}}.

\bibitem{kuperberg2019algorithmic}
G.~Kuperberg.
\newblock Algorithmic homeomorphism of 3-manifolds as a corollary of
  geometrization.
\newblock {\em Pacific J. Math.}, 301(1):189--241, 2019.
\newblock \href {https://doi.org/10.2140/pjm.2019.301.189}
  {\path{doi:10.2140/pjm.2019.301.189}}.

\bibitem{lackenby2004heegaard}
M.~Lackenby.
\newblock The {H}eegaard genus of amalgamated 3-manifolds.
\newblock {\em Geom. Dedicata}, 109:139--145, 2004.
\newblock \href {https://doi.org/10.1007/s10711-004-6553-y}
  {\path{doi:10.1007/s10711-004-6553-y}}.

\bibitem{lackenby2006heegaard}
M.~Lackenby.
\newblock Heegaard splittings, the virtually {H}aken conjecture and property
  {$(\tau)$}.
\newblock {\em Invent. Math.}, 164(2):317--359, 2006.
\newblock \href {https://doi.org/10.1007/s00222-005-0480-x}
  {\path{doi:10.1007/s00222-005-0480-x}}.

\bibitem{lackenby2017conditionally}
M.~Lackenby.
\newblock Some conditionally hard problems on links and 3-manifolds.
\newblock {\em Discrete Comput. Geom.}, 58(3):580--595, 2017.
\newblock \href {https://doi.org/10.1007/s00454-017-9905-8}
  {\path{doi:10.1007/s00454-017-9905-8}}.

\bibitem{lackenby2020algorithms}
M.~Lackenby.
\newblock Algorithms in 3-manifold theory.
\newblock In I.~Agol and D.~Gabai, editors, {\em Surveys in {$3$}-manifold
  topology and geometry}, volume~25 of {\em Surv. Differ. Geom.}, pages
  163--213. Int. Press Boston, 2020.
\newblock \href {https://doi.org/10.4310/SDG.2020.v25.n1.a5}
  {\path{doi:10.4310/SDG.2020.v25.n1.a5}}.

\bibitem{li2010heegaard}
Tao Li.
\newblock Heegaard surfaces and the distance of amalgamation.
\newblock {\em Geom. Topol.}, 14(4):1871--1919, 2010.
\newblock \href {https://doi.org/10.2140/gt.2010.14.1871}
  {\path{doi:10.2140/gt.2010.14.1871}}.

\bibitem{manolescu2016lectures}
C.~Manolescu.
\newblock Lectures on the triangulation conjecture.
\newblock In {\em Proc. 22nd {G}\"{o}kova {G}eom.-{T}opol. {C}onf. (GGT 2015)},
  pages 1--38. Int. Press Boston, 2016.
\newblock URL: \url{https://gokovagt.org/proceedings/2015/manolescu.html}.

\bibitem{maria2019treewidth}
C.~Maria and J.~Purcell.
\newblock Treewidth, crushing and hyperbolic volume.
\newblock {\em Algebr. Geom. Topol.}, 19(5):2625--2652, 2019.
\newblock \href {https://doi.org/10.2140/agt.2019.19.2625}
  {\path{doi:10.2140/agt.2019.19.2625}}.

\bibitem{markov1958homeo}
A.~Markov.
\newblock The insolubility of the problem of homeomorphy (in {R}ussian).
\newblock {\em Dokl. Akad. Nauk SSSR}, 121:218--220, 1958.

\bibitem{moise1952affine}
E.~E. Moise.
\newblock Affine structures in {$3$}-manifolds. {V}. {T}he triangulation
  theorem and {H}auptvermutung.
\newblock {\em Ann. Math. (2)}, 56:96--114, 1952.
\newblock \href {https://doi.org/10.2307/1969769} {\path{doi:10.2307/1969769}}.

\bibitem{orlik1967topologie}
P.~Orlik, E.~Vogt, and H.~Zieschang.
\newblock Zur {T}opologie gefaserter dreidimensionaler {M}annigfaltigkeiten.
\newblock {\em Topology}, 6:49--64, 1967.
\newblock \href {https://doi.org/10.1016/0040-9383(67)90013-4}
  {\path{doi:10.1016/0040-9383(67)90013-4}}.

\bibitem{poonen2014undecidable}
B.~Poonen.
\newblock Undecidable problems: a sampler.
\newblock In {\em Interpreting {G}\"{o}del}, pages 211--241. Cambridge Univ.
  Press, 2014.
\newblock \href {https://doi.org/10.1017/CBO9780511756306.015}
  {\path{doi:10.1017/CBO9780511756306.015}}.

\bibitem{raghavendra2010graph}
P.~Raghavendra and D.~Steurer.
\newblock Graph expansion and the unique games conjecture.
\newblock In {\em {P}roc. 2010 {ACM} {I}nt. {S}ymp. {T}heor. {C}omput.
  ({STOC}'10)}, pages 755--764. ACM, New York, 2010.

\bibitem{scharlemann2001comparing}
M.~Scharlemann and J.~Schultens.
\newblock Comparing {H}eegaard and {JSJ} structures of orientable 3-manifolds.
\newblock {\em Trans. Am. Math. Soc.}, 353(2):557--584, 2001.
\newblock \href {https://doi.org/10.1090/S0002-9947-00-02654-4}
  {\path{doi:10.1090/S0002-9947-00-02654-4}}.

\bibitem{scharlemann2005lecture}
M.~Scharlemann, J.~Schultens, and T.~Saito.
\newblock Lecture notes on generalized {H}eegaard splittings, 2005.
\newblock Open-access version of \cite{scharlemann2016lecture} with slightly
  different structure.
\newblock \href {http://arxiv.org/abs/math/0504167}
  {\path{arXiv:math/0504167}}.

\bibitem{scharlemann2016lecture}
M.~Scharlemann, J.~Schultens, and T.~Saito.
\newblock {\em Lecture Notes on Generalized {H}eegaard Splittings}.
\newblock World Scientific Publishing Co. Pte. Ltd., Hackensack, NJ, 2016.
\newblock \href {https://doi.org/10.1142/10019} {\path{doi:10.1142/10019}}.

\bibitem{scharlemann1992thin}
M.~Scharlemann and A.~Thompson.
\newblock Thin position for {$3$}-manifolds.
\newblock In {\em Geometric topology ({H}aifa, 1992)}, volume 164 of {\em
  Contemp. Math.}, pages 231--238. Am. Math. Soc., Providence, RI, 1994.
\newblock \href {https://doi.org/10.1090/conm/164/01596}
  {\path{doi:10.1090/conm/164/01596}}.

\bibitem{schultens2014introduction}
J.~Schultens.
\newblock {\em Introduction to $3$-Manifolds}, volume 151 of {\em Grad. Stud.
  Math.}
\newblock Am. Math. Soc., Providence, RI, 2014.
\newblock \href {https://doi.org/10.1090/gsm/151} {\path{doi:10.1090/gsm/151}}.

\bibitem{schultens2007destabilizing}
J.~Schultens and R.~Weidmann.
\newblock Destabilizing amalgamated {H}eegaard splittings.
\newblock In {\em Workshop on {H}eegaard {S}plittings}, volume~12 of {\em Geom.
  Topol. Monogr.}, pages 319--334. Geom. Topol. Publ., Coventry, 2007.
\newblock \href {https://doi.org/10.2140/gtm.2007.12.319}
  {\path{doi:10.2140/gtm.2007.12.319}}.

\bibitem{scott2014homeomorphism}
P.~Scott and H.~Short.
\newblock The homeomorphism problem for closed 3-manifolds.
\newblock {\em Algebr. Geom. Topol.}, 14(4):2431--2444, 2014.
\newblock \href {https://doi.org/10.2140/agt.2014.14.2431}
  {\path{doi:10.2140/agt.2014.14.2431}}.

\bibitem{waldhausen1967graph}
F.~Waldhausen.
\newblock Eine {K}lasse von {$3$}-dimensionalen {M}annigfaltigkeiten. {I},
  {II}.
\newblock {\em Invent. Math.}, 3:308--333; ibid. 4:87--117, 1967.
\newblock URL: \url{https://eudml.org/doc/141878} and
  \url{https://eudml.org/doc/141884}.
\newblock \href {https://doi.org/10.1007/BF01402956}
  {\path{doi:10.1007/BF01402956}}.

\bibitem{wu2014inapproximability}
Y.~Wu, P.~Austrin, T.~Pitassi, and D.~Liu.
\newblock Inapproximability of treewidth, one-shot pebbling, and related layout
  problems.
\newblock {\em J. Artificial Intelligence Res.}, 49:569--600, 2014.
\newblock \href {https://doi.org/10.1613/jair.4030}
  {\path{doi:10.1613/jair.4030}}.

\bibitem{yamazaki2018inapproximability}
K.~Yamazaki.
\newblock Inapproximability of rank, clique, {B}oolean, and maximum induced
  matching-widths under small set expansion hypothesis.
\newblock {\em Algorithms (Basel)}, 11(11):Paper No. 173, 10, 2018.
\newblock \href {https://doi.org/10.3390/a11110173}
  {\path{doi:10.3390/a11110173}}.

\end{thebibliography}

\appendix

\section{Selected Results from 3-Manifold Topology}
\label{app}

In this appendix we collect important facts about the topology of 3-manifolds we rely on. We do not aim to reproduce their complete proofs; indeed, those are often quite technical or include lengthy case analyses. However, we do aim to give precise statements and references.

\subsection{Weak reductions}
\label{app:weak}

\begin{example}\label{ex:weak-reduction} The \emph{$3$-torus} $\torus^3 = \nsphere{1}\times\nsphere{1}\times\nsphere{1}$ is a closed 3-manifold built by identifying opposite faces of a solid cube $K$ (\Cref{fig:3-torus}). It is well-known that $\torus^3$ has Heegaard genus $\heeg{\torus^3}=3$. A genus-three Heegaard splitting $\heegaard=(\hbody_1,\hbody_2,\surface)$ of $\torus^3$ can be obtained as follows: Upon the face-identifications, the 1-skeleton of $K$ becomes a bouquet \raisebox{-1.5pt}{\includegraphics[scale=.75]{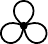}} of three circles. We let $\hbody_1 = \overline{N}_{\torus^3}(\raisebox{-1.5pt}{\includegraphics[scale=.75]{bouquet}})$ be a closed regular neighborhood of \raisebox{-1.5pt}{\includegraphics[scale=.75]{bouquet}} in $\torus^3$, $\hbody_2$ be the closure of $\torus^3\setminus\hbody_1$, and $\surface = \hbody_1 \cap \hbody_2$ be the resulting splitting surface (\Cref{fig:3-torus-splitting}).

\begin{figure}[ht]
\centering
     \begin{subfigure}[b]{0.3\textwidth}
         \centering
         \includegraphics[width=.85\textwidth]{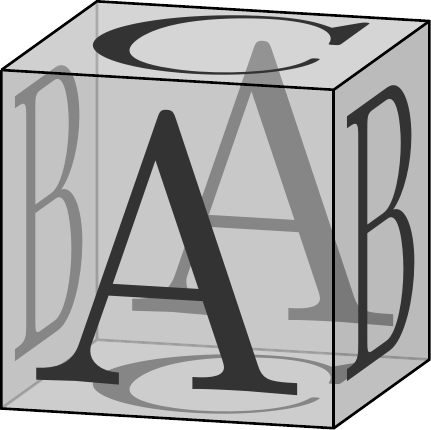}
         \caption{The 3-torus $\torus^3=\nsphere{1}\times\nsphere{1}\times\nsphere{1}$ can be obtained by identifying opposite faces of a solid cube.}
         \label{fig:3-torus}
     \end{subfigure}
     \hfill
     \begin{subfigure}[b]{0.3\textwidth}
         \centering
         \includegraphics[width=.85\textwidth]{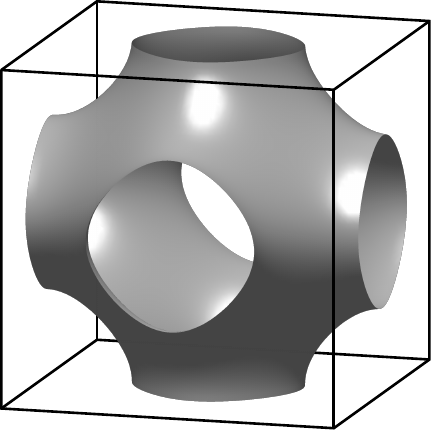}
         \caption{A surface $\surface \subset \torus^3$ defining a Heegaard splitting $(\hbody_1,\hbody_2,\surface)$ of $\torus^3$ of minimum genus three.}
         \label{fig:3-torus-splitting}
     \end{subfigure}
     \hfill
     \begin{subfigure}[b]{0.3\textwidth}
         \centering
         \begin{overpic}[width=.85\textwidth]{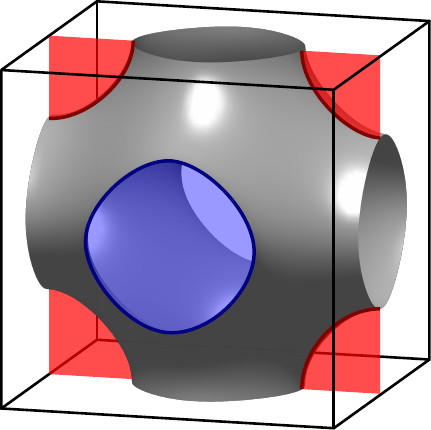}
         		\put (13,17) {\color[RGB]{0,0,0}$D_1$}
		\put (35,40) {\color[RGB]{0,0,0}$D_2$}
	\end{overpic}
         \caption{Two essential disks $D_i \subset \hbody_i$ $(i=1,2)$ witnessing the weak reducibility of $(\hbody_1,\hbody_2,\surface)$.}
         \label{fig:3-torus-weak-reduction}
     \end{subfigure}
        \caption{A weakly reducible minimum-genus Heegaard splitting of the 3-torus $\torus^3=\mathbb{S}^1\times\mathbb{S}^1\times\mathbb{S}^1$.}
        \label{fig:3-torus-example}
\end{figure}

Note that $D_1$ and $D_2$ in \Cref{fig:3-torus-weak-reduction} are essential disks in the handlebodies $\hbody_1$ and $\hbody_2$, respectively, and $\partial D_1 \cap \partial D_2 = \emptyset$, hence the splitting $\heegaard=(\hbody_1,\hbody_2,\surface)$ is weakly reducible. Now, a weak reduction at $\surface$ may be carried out as presented on \Cref{fig:schemata-weak-reduction}. We refer to \cite[p.\ 108--111]{scharlemann2016lecture} for details (cf.\ \cite[p.\ 61--62]{scharlemann2005lecture}).

\begin{figure}[!ht]
\centering
     \begin{subfigure}[b]{0.3\textwidth}
         \centering
         \begin{overpic}[scale=1.05]{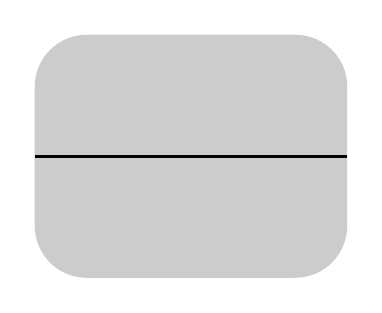}
		\put (1,39) {\small{$\surface$}}
		\put (46,23) {\small{$\hbody_2$}}
		\put (46,54) {\small{$\hbody_1$}}
		\put (102.5,54) {\LARGE{$\leadsto$}}
	\end{overpic}
         \label{fig:heegaard-splitting-schematic}
     \end{subfigure}
     \hfill
     \begin{subfigure}[b]{0.3\textwidth}
         \centering
         \begin{overpic}[scale=1.05]{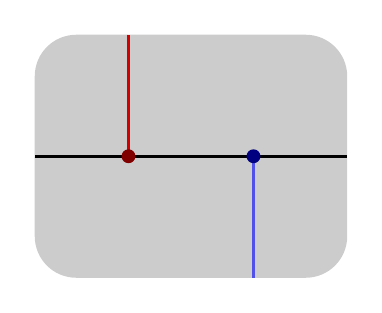}
		\put (1,39) {\small{$\surface$}}
		\put (62,0) {\small{$D_2$}}
		\put (60,46) {\small{$\partial D_2$}}
		\put (29.5,76.5) {\small{$D_1$}}
		\put (27.5,30.5) {\small{$\partial D_1$}}
		\put (102.5,54) {\LARGE{$\leadsto$}}
	\end{overpic}
         \label{fig:weak-reducing-disks}
     \end{subfigure}
     \hfill
     \begin{subfigure}[b]{0.3\textwidth}
         \centering
         \begin{overpic}[scale=1.05]{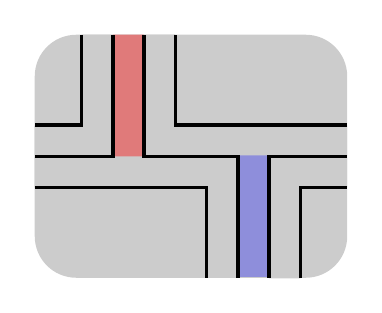}
	\end{overpic}
         \label{fig:weak-reduction-schematic}
     \end{subfigure}

	\vspace{-3pt}

     \begin{subfigure}[b]{1\textwidth}
		\begin{overpic}[scale=1.05]{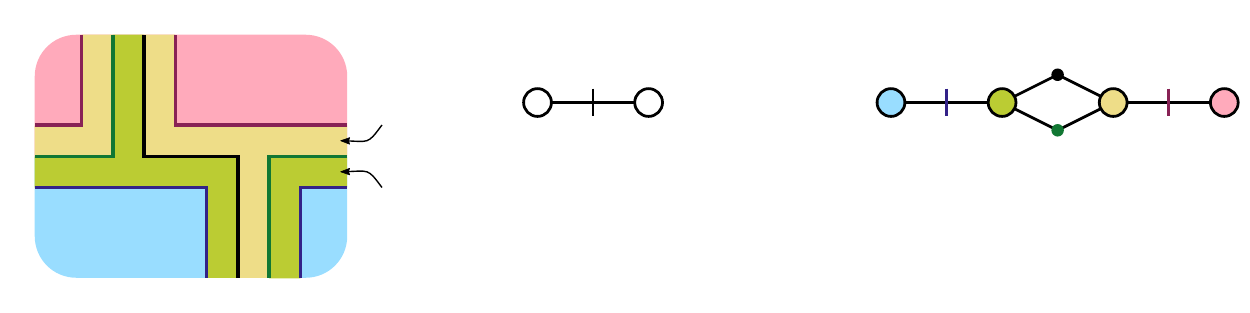}
			\put (46.75,19) {\small{$\surface$}}
			\put (42,12) {\small{$\hbody_2$}}
			\put (51,12) {\small{$\hbody_1$}}
			\put (74.75,19) {\small{$\surface_2$}}
			\put (70,12) {\small{$\compbody^{(2)}_1$}}
			\put (79,12) {\small{$\compbody^{(2)}_2$}}
			\put (92.5,19) {\small{$\surface_1$}}
			\put (87.75,12) {\small{$\compbody^{(1)}_1$}}
			\put (96.75,12) {\small{$\compbody^{(1)}_2$}}
			\put (37.5,4) {\small{$\width(\heegaard) = (g(\surface)) = (3)$}}
			\put (68.5,4) {\small{$\width(\heegaard') = (g(\surface_1), g(\surface_2)) = (2,2)$}}
			\put (60,16) {\Large{$\leadsto$}}
			\put (8.5,5.5) {\small{$\compbody^{(2)}_1$}}
			\put (19.25,18) {\small{$\compbody^{(1)}_2$}}
			\put (0,9) {\small{$\surface_2$}}
			\put (0,14.5) {\small{$\surface_1$}}
			\put (30.25,8) {\small{$\compbody^{(2)}_2$}}
			\put (30.25,15.5) {\small{$\compbody^{(1)}_1$}}
		\end{overpic}
		\label{fig:heegaard-splitting-schematic-colors}
	\end{subfigure}

	\vspace{-6pt}

	\caption{Schematic ``movie'' of a weak reduction of the genus-three Heegaard splitting of $\torus^3$.}
	\label{fig:schemata-weak-reduction}
\end{figure}

\end{example}

\subsection{Seifert fibered spaces and JSJ decompositions}
\label{app:jsj}

A key assertion about the manifold $\manifold_G$ constructed in \Cref{sec:const} is \Cref{claim:jsj}, according to which the dual graph $\dual(\jsjdecomp)$ of the JSJ decomposition $\jsjdecomp$ of $\manifold_G$ satisfies $\dual(\jsjdecomp) = G$. This can be shown through the same argument that is outlined by Lackenby in \cite[p.\ 591]{lackenby2017conditionally}. However, while we use the same building blocks as Lackenby (namely the manifolds $\altmanifold(k)$), we glue them together via different (\emph{high-distance}) maps. Therefore, some explanations are in order. To this end, closely following \cite{hatcher2007notes}, we recall some basic definitions and important results.

\subparagraph*{Seifert fiberings} 
A {\em model Seifert fibering} of the solid torus $\nsphere{1} \times D$ is a decomposition of $\nsphere{1} \times D$ into disjoint circles, called {\em fibers}, obtained as follows. Consider the solid cylinder $[0,1] \times D$ and identify its boundary disks $\{0\}\times D$ and $\{1\}\times D$ through a rotation by $2\pi p/q$, where $p$ and $q$ are relatively prime integers with $q > 0$. After this identification the segment $[0,1] \times \{0\}$ becomes a fiber $\nsphere{1}\times\{0\}$, whereas every other fiber consists of $q$ segments of the form $[0,1]\times\{x\}$. Next, a \emph{Seifert fibering} of  $3$-manifold $\manifold$ is a decomposition of $\manifold$ into disjoint circles (also called \emph{fibers}), such that each fiber has a neighborhood homeomorphic to that of a fiber in some model Seifert fibering of $\nsphere{1} \times D$ via a fiber-preserving homeomorphism. A $3$-manifold that admits a Seifert fibering is called a \emph{Seifert fibered space}. The following classical result says that most Seifert fibered spaces admit a unique Seifert fibering.

\begin{theorem}[\cite{orlik1967topologie,waldhausen1967graph}, adapted from {\cite[Theorem 2.3]{hatcher2007notes}}]
\label{thm:seifert-uniqueness}
Every orientable Seifert fibered space admits a unique Seifert fibering up to fiber-preserving homeomorphism, except for:
\begin{itemize}
	\item $\nsphere{1}\times D$, the solid torus (various model fiberings; see above),
	\item $\nsphere{1} \simtimes \nsphere{1} \simtimes I$, the orientable $\nsphere{1}$-bundle over the M\"obius band (two fiberings),
	\item Seifert fibered spaces over $\mathbb{R}P^2$ with one exceptional fiber (two fiberings),
	\item lens spaces $L(p,q)$ including $\nsphere{3}$ and $\nsphere{1} \times \nsphere{2}$ (various fiberings), and
	\item $\nsphere{1} \simtimes \nsphere{1} \simtimes \nsphere{1}$, the orientable $\nsphere{1}$-bundle over the Klein bottle (two fiberings).
\end{itemize}
\end{theorem}

\begin{corollary}
\label{cor:seifert-uniqueness-blocks}
The manifolds $\altmanifold(k)$ defined in \Cref{sec:const} admit unique Seifert fiberings.
\end{corollary}

\begin{proof}
The statement follows by observing that $\altmanifold(k)$ is a Seifert fibered space over the torus, and so it does not belong to any of the exceptional manifolds mentioned in \Cref{thm:seifert-uniqueness}.
\end{proof}

\begin{proposition}[folklore, see {\cite[Proposition 1.6]{hatcher2007notes}}]
\label{prop:irreducible-cover}
If an irreducible $3$-manifold $\widetilde{\manifold}$ is a covering space of the interior of another $3$-manifold $\manifold$, then $\manifold$ is irreducible as well.
\end{proposition}

\begin{claim}[folklore, see {\cite[p.\ 11]{hatcher2007notes}}]
\label{claim:surface-cover}
For any compact, connected and orientable surface $F$ (possibly with boundary) other than $\nsphere{2}$, the interior of $F \times \nsphere{1}$ has $\mathbb{R}^3$ as its universal cover.\footnote{The claim holds for non-orientable surfaces as well with the exception of $\mathbb{R}P^2$.}
\end{claim}

\begin{claimproof}
If $F \ncong \nsphere{2}$, it is well-known that the universal cover of $F$ is $\mathbb{R}^2$. Indeed, for $F \cong \torus^2$ a covering map $\mathbb{R}^2 \rightarrow F$ can be defined through the regular tiling of the Euclidean plane via squares, while for $F \cong \altsurface_g$ with genus $g > 1$ such a map can be obtained through the regular tiling of the hyperbolic plane via regular $4g$-gons, with $4g$ such polygons meeting at each vertex \cite{coxeter1956regular}.
If $F$ has non-empty boundary, then by the theory of covering spaces \cite[Theorem 10.19]{armstrong1983basic} the interior of $F$ has some universal cover $\mathscr{X}$, and by the \emph{uniformization theorem} \cite[Theorem XII.0.1]{saint-gervais2016uniformisation}, this space $\mathscr{X}$ -- being a non-compact, open, simply-connected surface -- must be homeomorphic to $\mathbb{R}^2$.
As the universal cover of $\nsphere{1}$ is $\mathbb{R}$, it follows that the universal cover of $F \times \nsphere{1}$ is $\mathbb{R}^3$.
\end{claimproof}

\begin{claim}[folklore, see {\cite[p.\ 14]{hatcher2007notes}}]
\label{claim:compressible-torus}
A $2$-sided torus $T \subset \manifold$ in an irreducible 3-manifold $\manifold$ is compressible if and only if $T$ either bounds a solid torus $\nsphere{1} \times D \subset \manifold$ or lies in a ball in $\manifold$.
\end{claim}

\begin{corollary}
\label{cor:irreducible}
For any $k \in \mathbb{N}$, $\altmanifold(k)$ is irreducible and has incompressible boundary.
\end{corollary}

\begin{proof}
The irreducibility of $\altmanifold(k)$ is a consequence of \Cref{prop:irreducible-cover} and \Cref{claim:surface-cover}. Then, the incompressibility of $\partial\altmanifold(k)$ follows from \Cref{claim:compressible-torus} by noting that $\partial\altmanifold(k)$ is a union of tori and, by construction of $\altmanifold(k)$, none of them bounds a solid torus or lies in a ball in $\altmanifold(k)$.
\end{proof}

\begin{proposition}
\label{prop:irreducible-MG}
The manifold $\manifold_G$ constructed in \Cref{sec:const} is irreducible.
\end{proposition}

\begin{proof}
Let $G=(V,E)$ be a multigraph on $n$ nodes with $m$ arcs, and $\mathbf{T} = \{T_1,\ldots,T_m\}$ be the collection of pairwise disjoint tori $T_i \subset \manifold_G$ that are the images of the boundary tori of the manifolds $\altmanifold_v \cong \altmanifold(d_v)$ after performing the torus-gluings in the construction of $\manifold_G$. 

Let us assume that $\manifold_G$ is \emph{not} irreducible, i.e., there is some reducing 2-sphere $S \subset \manifold_G$. By definition, $S$ does not bound a ball in $\manifold_G$. Note that $S$ cannot be entirely contained in any of the building blocks $\altmanifold_v \subset \manifold_G$, otherwise it would imply the reducibility of $\altmanifold_v \cong \altmanifold(d_v)$, contradicting \Cref{cor:irreducible}. It follows that $S \cap (\cup_{i=1}^m T_i) \neq \emptyset$. Thus, after an isotopy if necessary, we may assume that $S$ intersects $\cup_{i=1}^m T_i$ transversally, and the intersection $S \cap (\cup_{i=1}^m T_i)$ is a union of pairwise disjoint closed curves $\boldsymbol{\alpha}=\{\alpha_1,\ldots,\alpha_\ell\}$ with $\ell$ being minimal.

Now let $\alpha_j \in \boldsymbol{\alpha}$ be a curve that is \emph{innermost} in $S$, i.e., one of the two disks bounded by $\alpha_j$ in $S$, say $D \subset S$, does not contain any other curves from $\boldsymbol{\alpha}$. After relabeling if necessary, let $T_j \in \mathbf{T}$ denote the torus for which $\alpha_j = S \cap T_j$. Note that, if $\alpha_j$ also bounds a disk $D' \subset T_j$, then $D \cup D'$ is a sphere entirely contained in one of the building blocks $\altmanifold_w$ incident to $T_j$. Since $\altmanifold_w$ is irreducible (\Cref{cor:irreducible}), $D \cup D'$ bounds a ball $B \subset \altmanifold_w$, which can be used to eliminate $\alpha_j$ by first isotoping $D$ inside $B$ to coincide with $D'$ and then pushing $D'$ off the torus $T_j$, contradicting the minimality of $\boldsymbol{\alpha}$. Hence the disk $D$ is essential in $\altmanifold_w$. However, this contradicts the incompressibility of $T_j$ in $\altmanifold_w$ (\Cref{cor:irreducible}).
\end{proof}

\begin{proof}[Proof of {\Cref{claim:jsj}}]
We follow the notation introduced in the proof of \Cref{prop:irreducible-MG}. By \Cref{cor:seifert-uniqueness-blocks}, each building block $\altmanifold_v$ admits a unique Seifert fibering. In particular, each boundary component of $\altmanifold_v$ has a uniquely determined fibering. Therefore, when two building blocks $\altmanifold_{u}$ and $\altmanifold_{u'}$ are glued together along boundary tori $T \subset \partial\altmanifold_{u}$ and $T' \subset \partial\altmanifold_{u'}$, the resulting 3-manifold is not Seifert fibered, unless the gluing map aligns the (identical) Seifert fiberings of $T$ and $T'$. Now, it follows from the definition of distance \cite[Section 4.1]{bachman2017computing} that gluing maps of distance at least one do \emph{not} align these Seifert fiberings. (In fact, for any two fixed Seifert fiberings of the torus, there are precisely two maps that align them.) Hence, upon the manifold $\manifold_G$ is constructed, the Seifert fibrations of adjacent building blocks induce different fibrations on each torus in $\mathbf{T} = \{T_1,\ldots,T_m\}$.
Now, by construction of the building blocks, none of the tori in $\mathbf{T}$ bounds a solid torus (cf.\ \cite[p.\ 591]{lackenby2017conditionally}), therefore they are all incompressible (cf.\ \Cref{claim:compressible-torus}) and pairwise non-parallel. This implies that the building blocks $\altmanifold_v$ ($v \in V$) are indeed the JSJ pieces of $\manifold_G$ and the tori in $\mathbf{T}$ comprise its JSJ tori.
\end{proof}

\newpage

\section{Triangulating the building blocks}
\label{app:trg}

Here we describe the triangulation $\tri(k)$ of the building block $\altmanifold(k)$ introduced in \Cref{sec:const}. Topologically, $\altmanifold(k) = \torus^2_k \times \nsphere{1}$, where $\torus^2_k$ is obtained by removing the interiors of $k$ pairwise disjoint disks from the torus $\torus^2$. To construct $\tri(k)$, first we triangulate $\torus^2_k$ with $3k+2$ triangles as shown in \Cref{fig:torus+holes}. This triangulation of $\torus^2_k$ naturally lifts to a decomposition of $\torus^2_k \times [0,1]$ into $3k+2$ \emph{triangular prisms}, each of which can be triangulated with three tetrahedra, see \Cref{fig:trg_prism}. The construction of $\tri(k)$ is concluded by identifying the triangulations of $\torus^2_k \times \{0\}$ and $\torus^2_k \times \{1\}$ via the identity map. It immediately follows that $\tri(k)$ consists of $9k+6 = O(k)$ tetrahedra and induces a two-triangle triangulation at each boundary torus.

\begin{figure}[ht]
	\centering
	\includegraphics[scale=1]{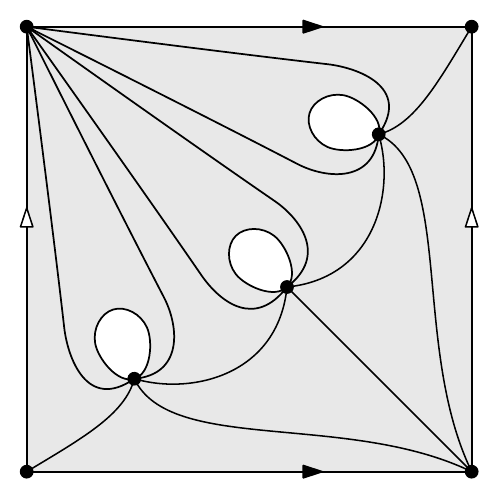}
	\caption{A minimal triangulation of $\torus^2_k$ with $3k+2$ triangles, drawn for $k =3$. The extension of the drawing for larger $k$ is straightforward.}
	\label{fig:torus+holes}
\end{figure}

\begin{figure}[ht]
	\centering
	\includegraphics[scale=0.909]{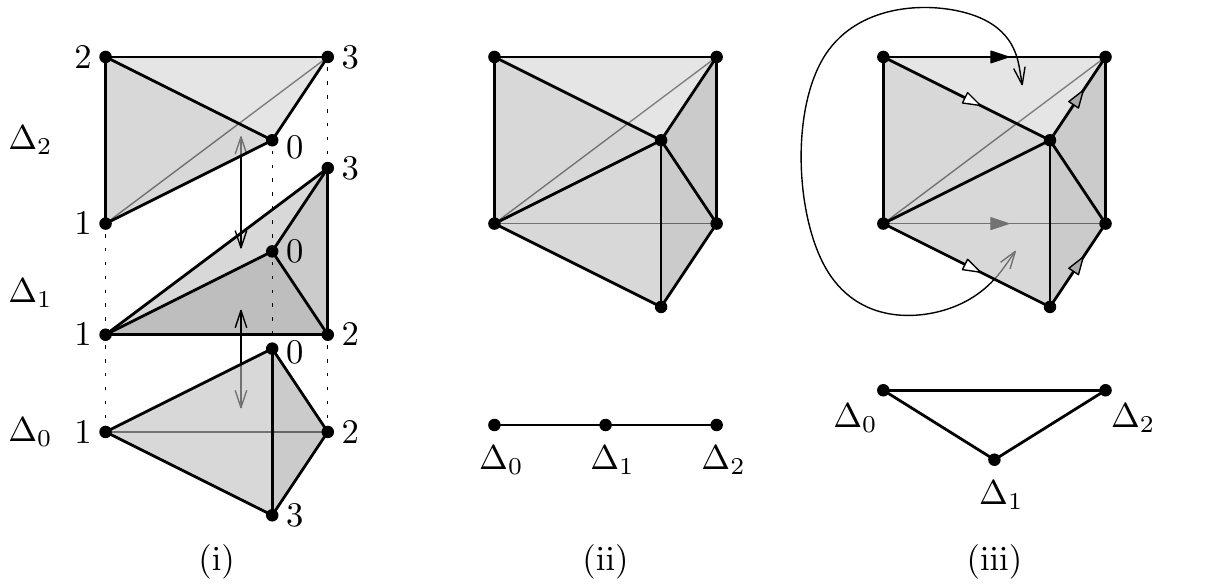}
	\caption{(i) The \emph{triangular prism} (the shape described as $t \times [0,1]$, where $t$ denotes a triangle) can be triangulated with three tetrahedra. (ii) A triangulation of the triangular prism and its corresponding dual graph. (iii) A triangular prism with its top and bottom triangles identified. Reproduced from \cite[Figure 13]{huszar2019manifold}.}
	\label{fig:trg_prism}
\end{figure}

\end{document}